\documentclass[12pt]{amsart}

\DeclareMathOperator{\cof}{\cof}

\usepackage[english]{babel}
\usepackage[utf8]{inputenc}
\usepackage{amsmath}
\usepackage{amssymb}
\usepackage{amsfonts}
\usepackage{amsthm}
\usepackage{mathrsfs}
\usepackage[all]{xy}
\usepackage[pdftex]{graphicx}
\usepackage{color}
\usepackage[pdftex,dvipsnames]{xcolor} 
\usepackage{cite}
\usepackage{url}
\usepackage{indent first}
\usepackage[labelfont=bf,labelsep=period,justification=raggedright]{caption}
\usepackage[english]{babel}
\usepackage[utf8]{inputenc}
\usepackage[colorlinks=true, allcolors=blue]{hyperref}
\usepackage[colorinlistoftodos]{todonotes}
\usepackage{tkz-fct}
\usepackage{mathdots}
\usepackage{comment}
\usepackage{float}

\usepackage{algpseudocode}
\usepackage{mathrsfs}
\usepackage{bbm}
\usepackage{dsfont}
\usepackage{theoremref}
\usepackage[dvipsnames]{xcolor}

\usetikzlibrary{automata,positioning}

\newcommand{\br}{\mathbb{R}}

\newcommand{\bz}{\mathbb{Z}}

\newcommand{\bs}{\mathbb{S}}

\newcommand{\frk}[1]{\mathfrak{#1}}

\newcommand{\defin}[1]{\begin{defn} {#1} \end{defn}}
\newcommand{\propos}[1]{\begin{prop} {#1} \end{prop}}

\newcommand{\coro}[1]{\begin{cor} {#1} \end{cor}}

\theoremstyle{plain}

\newcommand{\al}{\alpha}

\newcommand{\be}{\beta}

\newcommand{\s}{\sigma}

\newcommand{\la}{\lambda}

\newcommand{\eps}{\varepsilon}

\theoremstyle{plain}
\newtheorem{thm}{Theorem}

\newtheorem{lemma}[thm]{Lemma}
\newtheorem{cor}[thm]{Corollary}
\newtheorem{conj}[thm]{Conjecture}
\newtheorem{prop}[thm]{Proposition}

\newtheorem{qn}[thm]{Question}

\theoremstyle{definition}
\newtheorem{defn}[thm]{Definition}

\theoremstyle{remark}
\newtheorem{rem}[thm]{Remark}

\numberwithin{equation}{section}
\numberwithin{thm}{section}

\usepackage[margin=1in]{geometry} 

\numberwithin{equation}{section}
\numberwithin{thm}{section}

\begin{document}
\title[Non-Rigid Rank-One]{Non-Rigid Rank-One Infinite Measures on the Circle}

\author[Drillick]{Hindy Drillick}
\address[Hindy Drillick]{
     Columbia University, Department of Mathematics, New York, NY 10027, USA.}
\email{hdrillick@math.columbia.edu}

\author[Espinosa]{Alonso Espinosa-Dominguez}
\address[Alonso Espinosa-Dominguez]{Department of Mathematics\\
     Massachusetts Institute of Technology \\ Cambridge, MA 02139, USA.}
\email{aespdom@mit.edu}

\author[Jones]{Jennifer N. Jones-Baro}
\address[Jennifer N. Jones-Baro]{Department of Mathematics, Northwestern University, Evanston IL 60208}
\email{jenniferjones2024@u.northwestern.edu}

\author[Leng]{James Leng}
\address[James Leng]{Department of Mathematics\\
     University of California, Los  Angeles \\ Berkeley, CA 90095, USA.}
\email{jamesleng100@gmail.com}

\author[Mandelshtam]{Yelena Mandelshtam}
\address[Yelena Mandelshtam]{
     University of California, Berkeley, USA.}
\email{yelena@math.berkeley.edu}

\author[Silva]{Cesar E. Silva}
\address[Cesar E. Silva]{Department of Mathematics\\
     Williams College \\ Williamstown, MA 01267, USA.}
\email{csilva@williams.edu}


\maketitle
\begin{abstract}
  For a class of irrational numbers, depending on their Diophantine properties, we construct explicit rank-one transformations that are totally ergodic and not weakly mixing and classify when the measure is finite or infinite. In the finite case they are isomorphic to irrational rotations, giving explicit  rank-one cutting and spacer parameters for these irrational rotations. In the infinite case we use the constructions to provide  examples of non-weakly mixing infinite measure-preserving ergodic transformations which do not have any nontrivial probability preserving factors with discrete spectrum, thereby answering a questions of Aaronson and Nakada and of Glasner and Weiss.  We also obtain nonsingular versions of these examples for  each Krieger ration set.
\end{abstract}

\section{Introduction}

     In \cite{dJ76b}, del Junco proved  that (finite measure-preserving) discrete spectrum transformations are rank-one transformations; the main step in his proof was to show  that irrational rotations are rank-one. Rank-one transformations are transformations that are well-approximated by Rokhlin columns and have been a source of examples and counterexamples in ergodic theory; they are known to be generic in the group of finite measure-preserving transformations \cite{Fe97}, and  genericity in the infinite measure case  has  been verified recently   in \cite{BSSSW15}.  In a different paper, and before proving that irrational rotations are rank-one, del Junco showed that  a rotation by an irrational number that is well approximable is rank-one \cite{dJ76a}. In  \cite{dJ76a}, del Junco 
 also showed that for each  irrational $\alpha$ in a specific class  satisfying some approximation properties,  there is a rank-one transformation $T_\alpha$, with an explicit  cutting and stacking    construction, such that $T_\alpha$ has $e^{2\pi i\alpha}$ as an eigenvalue. His main goal was to show that these  irrational rotations  were factors  of rank-one transformations (this was before del Junco showed that factors of rank-one transformations 
are rank-one \cite[3.2]{dJ77}).  An interesting part of this proof is that the transformation $T_\alpha$ has an explicit cutting and stacking construction. 
Proofs that irrational rotations are rank-one such as del Junco's original proof \cite{dJ76b}, and later ones such as the one by Iwanic \cite{Iw94}, do not give explicit cutting and stacking parameters.

As is well known, there exist several definitions for rank-one transformations \cite{Fe97}. One constructive definition
uses a sequence of cutting parameters $(r_n)_{n\geq 0}$ and spacer parameters $s(i,j), i\geq 0, j=0,\ldots, r_n-1$. When such a sequence is given, which is known \emph{not} to be unique, and the transformation is defined on intervals,   we will say that there   is an explicit  cutting and stacking construction for the rank-one transformation. As far as we know, there is no algorithm for constructing the cutting and spacer parameters of a  transformation that is known to be rank-one by the abstract definition. 

In this paper we first extend del Junco's construction in \cite{dJ76a} to give, for each irrational $\alpha$ in a set of full measure, that we call not of golden type, an explicit cutting and stacking construction of a rank-one transformation $T_\alpha$  that depends on the continued fraction 
expansion of $\alpha$. In some cases the resulting rank-one transformation is infinite measure-preserving (this is different from del Junco's construction as he arranges it so that it is always on a space of finite measure). 
\thref{finitude} gives a condition depending on the irrational number for when the rank-one construction is defined  on a  finite or infinite measure space.  \thref{toterg} shows that all of these transformations are totally ergodic (rank-one transformations are ergodic but not necessarily totally ergodic).

One of our main results is that for each such irrational $\alpha$, if $f_\alpha$ is the eigenfunction of $T_\alpha$ corresponding to the 
eigenvalue $e^{2\pi i\alpha}$, then $f_\alpha $ is injective (\thref{injective}).  A consequence is that the rank-one transformation  $(T_\alpha,\mu)$ is isomorphic to rotation by $e^{2\pi i\alpha}$, $R_\alpha$, with the pushed measure $f_\alpha^*\mu$ on $\mathbb S^1$. For values of $\alpha$ for which $\mu $ is a  finite measure it follows that $f_\alpha^*\mu$ is  Lebesgue measure (or a multiple of),  since it is invariant for an irrational rotation; in this case we obtain explicit cutting and stacking constructions for the corresponding irrational rotation $R_\alpha$, i.e., we obtain cutting and spacer parameters for a rank-one construction that is isomorphic to an irrational rotation. In this context we mention a recent result in \cite{FGHSW} that classifies when a rank-one transformation is isomorphic to an odometer, or has an odometer factor, solely based on a condition on the cutting and spacer parameters of the rank-one transformation.     In the case when $\mu$ is infinite we obtain a Borel $\sigma$-finite nonatomic totally ergodic measure for an irrational rotation that is isomorphic to an explicit rank-one construction. 

We use this construction to answer questions in Aaronson--Nakada \cite{AaNa87} and  Glasner--Weiss  \cite{GlWe16}.  As is now well-known, many conditions that are equivalent to the weak mixing property for finite measure-preserving transformations do not necessarily remain equivalent in infinite measure; the reader may refer to \cite{AdSi18} for a survey of the weak mixing property for infinite measure-preserving and nonsingular transformations. In Section \ref{discretespectrum} we show  that there exist  nonsingular transformations, in fact infinite measure-preserving, that are ergodic and not weakly mixing but with no nontrivial factor that is finite measure-preserving with discrete spectrum, answering questions in Aaronson--Nakada \cite{AaNa87} and  Glasner--Weiss  \cite{GlWe16}. As is well-known, these examples  cannot exist in the finite measure-preserving case, see e.g., \cite{Ru90}.

We recall that the existence of infinite invariant measures on rotations have been known for a long time. 
Schmidt\cite{Sc77}  used Dye equivalence of group actions to show the existence of uncountably many $\sigma$-finite infinite nonatomic ergodic measures for irrational rotations. We 
construct explicit rank-one infinite measure-preserving transformations that we show are isomorphic to an irrational rotation with an infinite invariant measure. In \thref{rigid}, we show that $T_\alpha$ is rigid if and only if $\alpha$ is well approximable. In particular, for each $\alpha$ that is badly approximable we obtain a nonatomic $\sigma$-finite rank-one infinite invariant measure for $R_\alpha$ that is not rigid.  
This gives  infinite nonrigid ergodic, rank-one, invariant measures for some irrational rotations (for badly approximable $\alpha$).

We also consider  nonsingular versions of this construction. In fact, for each $\lambda\in [0,1]$, we construct 
a type III$_\lambda$ rank-one transformation; this yields nonsingular  type III$_\lambda$ rank-one  
rank-one measures on irrational rotations. We recall that Keane \cite{Ke71} constructed, for each irrational rotation, uncountably many inequivalent nonsingular measures on the circle.   We obtain a refinement of that result by showing that we obtain such 
a nonsingular measure for each Krieger type III$_\lambda$. We note that El Abdalaoui  has informed the authors that using methods from \cite{ElNa16} and his work he can obtain measures that  are spectrally disjoint for different $\alpha$,
and he has remarked to us that 
Keane's construction was generalized by Host, M\'ela, and Parreau \cite{HMP86}. For further information on eigenvalues of transformations we refer to \cite{Na98}.

\textbf{Acknowledgments:} This paper is based on research by the ergodic theory group of the 2018 SMALL undergraduate research project at Williams College. Support for the project was provided by National Science Foundation grant DMS-1659037 and the Science Center of Williams College. We thank Steven J. Miller for bringing   \thref{gausscor} to our attention, and Terrence Adams for suggesting reference \cite{AOW85}.   C. S.  would like to thank  El H. El Abdalaoui  for delightful visits to the University of Rouen Normandy where we discussed nonsingular transformations and Keane's examples, and for comments on an earlier version of this paper; C.S.  also benefitted from discussions on rank-one transformations with M. Foreman, S. Gao, A. Hill, and B. Weiss as part of a SQuaRe program at AIM. Since August 2019 until August 2021, C.S. has been serving as a Program Director in the Division of Mathematical Sciences at the National Science Foundation (NSF), USA, and as a component of this job, he received support from NSF for research, which included work on this paper. Any opinion, findings, and conclusions or recommendations expressed in this material are those of the authors and do not necessarily reflect the views of the National Science Foundation.

\section{Preliminaries}

Throughout this paper  $(X, \frk B, \mu)$ will  be a $\sigma$-finite, nonatomic standard Borel measure space.
We will assume all of our transformations are measurable and invertible.  A (measurable, invertible) transformation $T: X \to X$ is \textbf{measure-preserving} if for all $A \in \frk B$, we have $\mu(T^{-1}(A)) = \mu(A)$, and  \textbf{nonsingular} if for all $A \in \frk B$, $\mu(T^{-1}(A)) = 0$ if and only if $\mu(A) = 0$. A nonsingular transformation is \textbf{ergodic} if $T^{-1}A=A$ implies $\mu(A)=0$ or $\mu(A^c)=0$. $T$ is \textbf{totally ergodic} if $T^n$ is ergodic for all integers $n\neq0$. As our measures are nonatomic and the transformations are invertible, ergodic implies conservative (i.e., for all $A$ of positive measure we have $\mu(T^{-n}A\cap A)>0$ for some $n>0$). Finally, if $X\subset \br$, we will work exclusively with Lebesgue measure on the real line, and denote this measure by $m$.

\subsection{Rank-One Transformations}
\hfill \break
As remarked earlier, rank-one transformations play a central role in ergodic theory and this concept is central to our paper. We proceed to provide the definition of rank-one that we will use.

\defin{
A \textbf{column} of $T$ is a finite ordered collection of subsets $$C_k=\{C_k(0), C_k(1),\dots,C_k(h_k-1)\}$$ of $X$ of finite measure such that $T(C_k(i))=C_{k}(i + 1)$ for $0\leq i<h_k-1$. $C_k(0)$ is called the \textbf{base} of the $k$th column, $C_k(i)$ is the $i$-th \textbf{level} of the $k$th column, and $h_k$ is the \textbf{height} of the $k$th column.  }

\defin{
A transformation $T$ is \textbf{rank-one} if there exists an an ordered collection of columns $\{C_k\}_{k=1}^{\infty}$ and $\{C_{k}(0), C_{k}(1), ..., C_{k}(h_k-1)\}$ its levels, such that for every measurable set $A\subset X$ with positive nonzero measure and every $\varepsilon>0$ there exists  $N\neq0$ such that for every $j\geq N$ there exists a set $\widehat{C}_j$, a union  of some levels of $C_j$, such that
$$\mu(\widehat{C}_j\triangle A)<\varepsilon.$$
}

A very nice property of rank-one transformations is that they are isomorphic to transformations which can be explicitly constructed via the so called \textbf{cutting and stacking} method. Famous examples of cutting and stacking transformations include the Chac\'on and Kakutani transformations. See \cite{Si08} for a discussion of these and other examples.

The following follows from \cite{AOW85}. 

\propos{All rank-one transformations have an isomorphic representation through cutting and stacking.
}
\subsection{Weak Mixing}

\defin{\thlabel{eigendef} Let $T$ be a nonsingular  transformation.  A number $\lambda \in \mathbb{C}$ is an L$^\infty$ \textit{eigenvalue} of $T$ if there exists a nonzero a.e. function $f \in L^{\infty}$ such that 
\[f(T(x)) = \lambda f(x) \, \text{ a.e.}.\]
$f$ is called an \textbf{eigenfunction} of $T$. An eigenvalue $\lambda$ is a \textbf{rational eigenvalue} if there exists an $n \neq 0$ such that $\lambda^n = 1$. Otherwise, $\lambda$ is an \textbf{irrational eigenvalue}.
}
It can be shown that eigenvalues have modulus $1$. Also, when $T$ is ergodic its eigenfunctions have constant modulus, so when the measure is finite the eigenfunctions are in $L^2$. \\

 The following is well known (see e.g. \cite{Si08} for the outline of an argument that also works in the nonsingular case). 

\begin{prop}
An ergodic transformation $T$ is totally ergodic if and only if it has no rational eigenvalues other than $1$.
\end{prop}

\defin{Let $T$ be a nonsingular transformation. We say 
$T$ is \textbf{weakly mixing} if all eigenfunctions for $T$ are constant a.e.. This is equivalent to $T$ being ergodic 
and having $1$ as its only eigenvalue.
}

In the finite measure-preserving case there are several equivalent characterizations of weak mixing (see e.g., \cite{Ru90, Si08}).
In the infinite measure-preserving and nonsingular cases the situation is quite different and the reader may refer to 
\cite{AdSi18,GlWe16}. A weakly mixing transformation is totally ergodic. 

\subsection{Diophantine Approximations}

\hfill \break
 As is well known, every irrational number $\alpha$ can be uniquely described as a continued fraction $[a_0; a_1, a_2,...]$ where the $a_k$ are called the \textbf{coefficients} of the continued fraction. The rational numbers $\frac{p_k}{q_k} := [a_0; a_1, ..., a_{k-1}] $ are called the \textbf{convergents} for $\alpha$. The sequence of convergents provides a best Diophantine approximation of $\alpha$ by rational numbers. We can express the convergents in terms of the coefficients by the following recursive formulas, for $k\geq 2$,
\begin{align} 
    \nonumber p_k = a_{k-1}p_{k-1} + p_{k-2}, \hspace{0.5in} &p_0 = 1,   p_1 = a_0 \\
    q_k = a_{k-1}q_{k-1} + q_{k-2}, \hspace{0.5in}  &q_0 = 0, q_1 = 1.\label{b}
\end{align}

\begin{prop}\thlabel{a}
For each $n \ge 1$, if $\frac{p_n}{q_n} > \alpha$, then $\frac{p_{n + 1}}{q_{n + 1}} < \alpha$ and if $\frac{p_n}{q_n} < \alpha$, then $\frac{p_{n + 1}}{q_{n + 1}} > \alpha$. 
\end{prop}

\defin{An irrational number $\alpha$ is \textbf{badly approximable} if its continued fraction coefficients $(a_k)$ are bounded. Numbers that are not badly approximable are called \textbf{well approximable}.
}

\begin{prop}
The set of badly approximable numbers has measure zero.
\end{prop}

\defin{An irrational $\al$ is of \textbf{golden type} if there exists an $n$, such that for all $k \geq n$, $a_k= 1$.}
The golden ratio is of golden type. Observe that the set of golden type numbers is a countable subset of the badly approximable numbers.

\begin{lemma}\thlabel{irrLem}
Suppose $\al$ is an irrational number. Let $\left(\frac{p_k}{q_k}\right)$ be its continued fraction convergents, and define  $ \eps_k := \left|\al- \frac{p_k}{q_k}\right|$. Then 
$$\sum_{k=1}^\infty \eps_kq_{k+1}<\infty.$$
\end{lemma}
\begin{proof}
By \thref{a} the fractions $\left( \frac{p_k}{q_k}\right)$ alternate being greater and less than $\al$. Thus, we have
\begin{align*}
    \eps_k &= \left|\alpha - \frac{p_k}{q_k}\right|\\
    & < \left|\frac{p_{k+1}}{q_{k+1}} - \frac{p_k}{q_k}\right|\\
    & = \left| \frac{p_{k+1}q_k - p_kq_{k+1}}{q_{k+1}q_k} \right|\\
    & = \frac{1}{q_kq_{k+1}}
\end{align*}
by properties of continued fractions.
Therefore, $\sum \eps_kq_{k+1} \leq \sum \frac{1}{q_k} < \infty$ since the $q_k$ grow exponentially.
\end{proof} 

\section{Survey of results}\label{survey}

Our central result can be summarized in the following theorem.

\begin{thm}\thlabel{mainthm}
Suppose $\al$ is an irrational number that is  not of golden type. Then we  explicitly construct a rank-one transformation $T_\al$ on a Borel set $X_\al$ in $\br$  with the following properties.
\begin{enumerate}
    \item There is an $L^\infty$ eigenfunction $f_\al$ with eigenvalue  $e^{2\pi i \al}$. Hence, $T_\al$ is not weakly mixing.
    \item The eigenfunction is injective almost everywhere.
    \item  $T_\al$ is totally ergodic.
    \item If $X_\al$ has infinite measure and $\al$ is badly approximable, then every probability preserving factor of $T_\al$ is weakly mixing.
    \item $T_\al$ is nonrigid for badly approximable numbers, and rigid otherwise.
\end{enumerate}
Moreover, there exist uncountably many $\al$ so that $X_\al$ has finite measure, but for Lebesgue almost every irrational $\al$, $X_\al$ is an infinite measure subset of $\br$.  Thus we obtain a type $II_\infty$ transformation with properties (1)-(3).
\end{thm}

The proof of this theorem is subdivided as follows: in Section \ref{construction}, we construct our transformation and an eigenfunction. Section \ref{finiteOrNot} characterizes when $X_\al$ will have infinite measure and when it will not, and shows that the former occurs for almost every $\al$. Total ergodicity is proven in Section \ref{toteserg}.  Injectivity of $f_\al$ is proven in Section \ref{one-one}, and probability preserving factors are considered in Section \ref{discretespectrum}.  Rigidity questions are taken up in \ref{badlyrigid}.\\\\

Notice that injectivity of $f_\al$ and Blackwell's theorem tell us that $f_\al^{-1}(\mathfrak{B}(\bs^1)) = \mathfrak{B}(X_\al)$.
As such, $\nu = m\circ f_\al^{-1}$ is a measure on $\bs^1$. Moreover, letting    
$$S_0 = \bigcap_{n\in \bz}R^n_\al f_\al(X_\al),\qquad X_0= f_\al^{-1}(S_0),$$
then $\nu(\bs^1\setminus S_0) = 0 = m(X_\al\setminus X_0)$, $f_\al$ maps $X_0$ bijectively onto $S_0$, and $f_\al(T_\al x) = e^{2\pi i \al} f_\al(x) = R_\al(f(x))$. Thus, $f$ is an isomorphism between the dynamical systems $(X_\al, m, T_\al)$ and $(\bs^1, \nu, R_\al)$. \\\\
In other words, letting $\s$ be the standard measure on the circle and $\al$ be such that $X_\al$ has finite measure, then by unique ergodicity of $(\bs^1, \s, R_\al)$, $\nu$  is equal to $\s$ up to a constant multiple, and thus we have explicitly constructed a cutting and stacking representation for an irrational rotation. If $X_\al$ has infinite measure, we have obtained an infinite-measure preserving dynamical system $(\bs^1, \nu, R_\al)$ that shares many of the properties of $(\bs^1, \s, R_\al)$; namely, it is not weakly mixing, it is totally ergodic, it has $e^{2\pi i n\al}$ as eigenvalues and $e_n:z\mapsto z^n$ as eigenfunctions for all $n\in \bz$.\\\\
We are interested in understanding what other properties of the rotation with Lebesgue measure carry over to $(\bs^1, \nu, R_\al)$ when $\nu$ is infinite. In section \ref{unique}, we prove that for badly approximable $\al$, $T_\al$ has only $e^{2\pi i\al}$ as eigenvalues, and thus the same can be said about the associated dynamical system on the circle. Whether this is true of $T_\al$ for $\al$ that are not badly approximable remains open. Using these properties of $T_\al$, we show in Section \ref{discretespectrum} that if $X_\al$ has infinite measure and if $\alpha$ is badly approximable, every probability preserving factor of $T_\al$ is weakly mixing. This implies that every nontrivial probability preserving factor of $X_\al$ does not have discrete spectrum, thereby providing a solution to Problem 5.5 of \cite{GlWe16}. \\\\
Another property that we analyze is rigidity. In section \ref{badlyrigid}, we show that $T_\al$ is rigid if and only if $\al$ is not badly approximable. Hence, for badly approximable $\al$, the system $(\bs^1, \nu, R_\al)$ differs in an important qualitative way to rotation with Lebesgue measure, even though both systems share many other properties, including having the same eigenvalues and eigenfunctions.\\\\
Finally, Section \ref{type3} extends our construction to the case of  type $III_\la$ nonsingular transformations.

\section{The construction}\label{construction}
Given an irrational $\alpha$ that is not of golden type, we construct a rank-one transformation $T_\alpha$ with $e^{2\pi in\al}$ as an eigenvalue using the cutting and stacking procedure. $T_\alpha$ will be defined on subsets of $\br$ with Lebesgue measure $m$. As mentioned in the introduction this is based on and extends del Junco's construction \cite{dJ76a}.  For the remainder of the paper we assume $\alpha $ is an irrational number   that  is not of golden type.

Recall that $(a_k)$ and $(\frac{p_k}{q_k})$ are the coefficients and convergents of the continued fraction expansion of $\alpha$. We have  by  formulas \ref{b} that $q_{k-1}= q_{k+1}-q_ka_k$, and $q_1 =1$. We start with column $C_1\subset \br$ of height $q_1 = 1$. From here on out, until specified otherwise, all of our intervals will be half open intervals: closed on the right and open on the left. For our purposes, it does not matter exactly what interval is in $C_1$ is, however, for convenience we take it to be the unit interval. Also, for notational convenience we let $C_k$ denote both the column and the union of the levels in the column, and it should be clear from the context which definition is meant. 
Suppose that the column $C_k$ with height $q_k$ has been constructed. Then we construct $C_{k + 1}$ to have height $q_{k+1}$ as follows: We cut $C_k$ into $a_k$ sub-columns of equal width and stack these above each other into a single column. We then add $q_{k+1}-q_ka_k = q_{k-1}$ additional levels on top so that the final height is $q_{k + 1}$ (see Figure \ref{fig}). These  additional levels, also called ``spacers",  are taken from $\br\setminus C_k$.  
Let $X$ be the union of all the columns. Notice that, setting $\mu_k:=m(C_k)$, we have
$$\mu_{k + 1} = \frac{q_{k-1}}{a_kq_k}\mu_k + \mu_k,$$
and hence 
\begin{equation}\label{measureX}
m(X) = \mu_1\prod_{k=2}^\infty \left(\frac{q_{k-1}}{a_kq_k} + 1\right). 
\end{equation}
Depending on the $(q_k)$, we may end up with a finite or an infinite measure space $X$. A necessary and sufficient condition for \ref{measureX} to converge is that $\sum_{k=1}^\infty \frac{q_{k-1}}{a_k q_k}$ converges. Let $T_\alpha: X \to X $ be the rank-one transformation defined by the above cutting and stacking procedure. We now see why it is necessary that $\alpha$ not be of golden type. (In that case, $a_k = 1$ for $k$ large enough, and we would no longer be cutting our columns. $T_\alpha$ would not be rank-one; in fact, it is not even surjective on $X$.)\\

\begin{figure}[h]
\begin{tikzpicture}
\draw [ very thick] (1,0) -- (1,5);
\draw [ very thick] (2,0) -- (2,5);
\draw [ very thick] (3,0) -- (3,5);
\draw [ very thick] (5,0) -- (5,5);

\draw [blue, very thick] (0,0) -- (6,0);
\draw [blue, very thick] (0,1) -- (6,1);
\draw [blue, very thick] (0,2) -- (6,2);
\draw [blue, very thick] (0,3) -- (6,3);
\draw [blue, very thick] (0,5) -- (6,5);

\draw [red, very thick] (5,5.5) -- (6,5.5);
\draw [red, very thick] (5,6) -- (6,6);
\draw [red, very thick] (5,6.5) -- (6,6.5);
\draw [red, very thick] (5,7.5) -- (6,7.5);

\draw [decorate,decoration={brace,amplitude=7pt, mirror},xshift=0pt,yshift=4pt]
(0,-0.5) -- (6,-0.5) node [black,midway,yshift=-0.6cm] 
{\footnotesize $a_{k}$};

\draw [decorate,decoration={brace,amplitude=7pt},xshift=0pt,yshift=4pt]
(-0.5,-0.2) -- (-0.5,5) node [black,midway,xshift=-0.7cm] 
{\footnotesize $q_{k}$};

\draw [decorate,decoration={brace,amplitude=7pt, mirror},xshift=0pt,yshift=4pt]
(6,5) -- (6,7.5) node [black,midway,xshift=0.8cm] 
{\footnotesize $q_{k-1}$};

\node (A) at (4, 4) [] {$\iddots$};
\node (A) at (5.5, 7.1) [] {$\vdots$};
\node (A) at (4, 0.5) [] {$\dots$};
\node (A) at (4, 1.5) [] {$\dots$};
\node (A) at (4, 2.5) [] {$\dots$};

\end{tikzpicture}
\caption{}\label{fig}
\end{figure}

We will now construct an eigenfunction $f\in L^\infty(X)$ for $T_\alpha$ with $\la = e^{2\pi i \al}$ as an eigenvalue. 
First define functions $g_k:X\to \mathbb{C}$ by setting $g_k = \lambda^n$ on the $n$th level of $C_k$ and $g_k = 0$ on $X \setminus C_k$. Now define $f_k$ by setting $f_k = g_k$ on $C_k$, and $f_k = g_{j}$ on $C_{j}\setminus C_{j-1}$ for all $j\geq k+1$. Then the $f_k$ are defined on all of $X$ and are clearly in $L^\infty$.  We would now like to show that the $f_k$ converge under the $L^\infty$ norm to an $f\in L^\infty(X)$. To do so, we bound $|f_{k+1}- f_{k}|$. For this, note that on $X\setminus C_k$, $f_k= f_{k+1}$, and so we need to worry only about the difference on $C_k$. $C_{k+1}$ is built by cutting the levels of $C_k$ into $a_k$ columns of equal size, and so we will let $C_k^j(n)$ denote the $n$th level in the $j$th new column created by cutting $C_k$ in this fashion, where $0\leq j \leq a_{k}-1$. Observe that on $C_k^j(n)$, $f_{k+1} = \la^{jq_k+n}$, and $f_k = \la^n$, giving that for each $0\leq j\leq a_{k}-1$,
\begin{align*}|f_{k+1}-f_k| &= |\la^n(\la^{jq_k}-1)|\\
& = |\la^{jq_k}-1|\\
&= |e^{2\pi i(\eps_k \pm \frac{p_k}{q_k})jq_k}-1|\\
& \leq 2\pi\eps_kj q_k \text{ \hspace{1.09in} on $C_k^j(n)$} \\
& < 2\pi\eps_ka_kq_k \text{ \hspace{1in} on $C_k(n)$ since $j<a_k$.}
\end{align*}
This tells us that $\|f_{k+1}-f_k\|_\infty < 2\pi\eps_ka_kq_k \leq 2\pi\eps_k q_{k+1}$ since $a_kq_k = \lfloor q_{k+1}/q_k\rfloor q_k \leq q_{k+1}$. Therefore, by \thref{irrLem},
$$\sum_{k = 0}^\infty ||f_{k + 1} - f_k||_{\infty} < \infty,$$
 making $(f_k)$ a  Cauchy sequence. We let $f = \lim_{k \to \infty} f_k$. Now we confirm that $\la$ is an eigenvalue of $f$. Given $x \in X$, we can find some column $C_K$ so that $x \in C_K$ and the level containing $x$ is not the top level of $C_K$. Then for all $k \geq K$ we have that $f_k(T_\al(x)) = \la f_k(x)$. Then $$f(T_\al(x)) = \lim_{k \to \infty}f_k(T_\al(x)) = \la \lim_{k \to \infty}f_k(x) = \la f(x).$$
 Thus, we have that $f$ is indeed an eigenfunction with $\al$ as an eigenvalue for the transformation $T_\al$.
 
 \begin{rem}
If  $X$ has finite measure, then $L^\infty(X)\subset L^2(X)$ and $(f_k)$ converges in the $L^2$ norm as well; hence the above procedure constructs an eigenfunction in the appropriate function space.  
 \end{rem}
 \begin{rem}
In the case when $X$ has finite measure our construction is in fact a special case of del Junco's construction in \cite{dJ76a}.
 However, in our case, by \thref{relprimecols}, $T_\alpha$ is always totally ergodic, but in general del Junco's construction need not be. Therefore  there exist del Junco constructions as in \cite{dJ76a} which cannot be isomorphic to an irrational rotation.

 \end{rem}

\section{Measure of $X_\al$}\label{finiteOrNot}

We now classify irrational numbers based on whether $X_\alpha$ as constructed in Section \ref{construction} will have finite or infinite measure. This is the same as asking whether the total measure of the spacers added is finite or infinite. 

\begin{thm}\thlabel{finitude}
The transformation $T_\alpha$ is an infinite measure-preserving transformation if and only if $\sum_{k=2}^\infty \frac{1}{a_ka_{k-1}} = \infty$.
\end{thm}

\begin{proof} 
We know by equation \ref{measureX} that $T_\al$ is an infinite measure transformation if and only if $ \sum_{k=1}^\infty \frac{q_{k-1}}{a_kq_k} = \infty$. By formulas \ref{b}, 
$$a_{k-1}q_{k-1} < q_k = a_{k-1}q_{k-1} + q_{k-2} < 2a_{k-1}q_{k-1}$$
Therefore, 
\begin{align*}
    \sum_{k=2}^\infty \frac{1}{2a_ka_{k-1}} &= \sum _{k=2}^\infty \frac{q_{k-1}}{a_k(2a_{k-1}q_{k-1})} \\
    &\leq \sum_{k=2}^\infty \frac{q_{k-1}}{a_{k-1}q_k} \\
    &\leq \sum_{k=2}^\infty \frac{q_{k-1}}{a_k(a_{k-1}q_{k-1})} \\
    &= \sum _{k=2}^\infty \frac{1}{a_ka_{k-1}}.
\end{align*}
Then $\sum_{k=2}^\infty \frac{q_{k-1}}{a_{k-1}q_k}$ diverges if and only if $\sum_{k=2}^\infty \frac{1}{a_ka_{k-1}}$ diverges, and so we are done.
\end{proof}

Since badly approximable numbers are precisely the irrational numbers with bounded coefficients $a_k$, it is clear that for a badly approximable $\alpha$, $\sum_{k=2}^\infty \frac{1}{a_ka_{k-1}}$ diverges, and $T_\alpha$ will be an infinite measure-preserving transformation. On the other hand, an example of irrationals such that the sum converges are any numbers whose continued fraction coefficients $(a_k)$ are strictly increasing. A specific example is the number $\alpha$ with continued fraction coefficients $[0; 1, 2, 3, 4...] = \frac{I_1(2)}{I_0(2)}$, where $I_1, I_0$ are the modified Bessel functions of the first kind. In that case, $T_\alpha$ will act on a finite measure space.

\begin{cor}
There are uncountably many $\alpha$ for which $X_\alpha$ has finite measure.
\end{cor}
\begin{proof}
Letting $a_k$ equal to $2^k$ or $2^k + 1$, we find that there are uncountably many sequences $(a_k)$ such that
$$\sum_k \frac{1}{a_ka_{k - 1}} < \infty.$$
Each sequence $(a_k)$ defines a unique irrational number, so there are uncountably many irrational numbers $\alpha$ such that $X_\alpha$ has finite measure.
\end{proof}

Now we address the question of the measure of the set of all irrational numbers such that $\sum_{k=2}^\infty \frac{1}{a_ka_{k-1}}$ converges. We use a corollary of the famous theorem due to Gauss and Kuzmin to answer this question. We will first want to compute the probability that for $\alpha\in \mathbb{R}$ the $k$th continued fraction coefficient $a_k$ of $\alpha$ equals $n$. Without loss of generality we will assume for the remainder of this section that $\alpha \in [0, 1)$. This is because the continued fraction expansion of $\alpha$ is identical the continued fraction expansion of $\alpha \pmod{1}$ with the exception of the first coefficient. We can consider the standard Lebesgue measure on $[0, 1)$, then the $k$th continued fraction coefficient $a_k(\alpha)$ is a well defined random variable. 

\begin{thm}[Gauss-Kuzmin] \thlabel{gauss}
Let $P(a_k(\alpha) = n)$ be the probability that the $k$th continued fraction coefficient $a_k$ of $\alpha$ equals $n$. Then $$\lim_{k\to\infty} P(a_k(\alpha) = n) = \log_2\left(1 + \frac{1}{n(n+2)}\right).$$
\end{thm}

The following corollary of the proof of the Gauss-Kuzmin theorem (see \cite[Theorem 10.4.2]{MT06}) shows that the probability that a coefficient takes on a certain value is dependent on the value of the previous coefficient. 
\begin{cor}\thlabel{gausscor}
$$\lim_{k \to \infty} P(a_{k}(\alpha) = n_2 | a_{k-1}(\alpha) = n_1)$$ 
\begin{equation} \thlabel{prob}
    = \frac{\log\left(1+ \frac{1}{[(n_2+1)n_1+1][(n_1+1)n_2+1]}\right)}
    {\log\left(1+ \frac{1}{n_1(n_1+2)}\right)}
\end{equation}

\end{cor}   

By \thref{gauss}, we see that $P(a_k = 1) \to \log_2(\frac{4}{3}) \approx 0.4150$. 
We can similarly evaluate \thref{prob} to see that $P(a_k = 1|a_{k-1} = 1) \to 0.3662$. Therefore the probability that a pair of consecutive coefficients $a_k, a_{k+1}$ both equal  $1$ approaches $0.4150 \cdot 0.3662$, as $k$ gets large. Since this is a nonzero probability, such pairs will occur infinitely often in the continued fraction expansion of almost all $\alpha$, thus implying that $\sum_{k=0}^\infty \frac{1}{a_ka_{k-1}}$ diverges. Therefore, the answer to our question is that $T_\alpha$ has infinite measure for almost all $\alpha$.

\section{$T_\al$ is totally ergodic}\label{toteserg}

We show that $T_\al$ is totally ergodic, i.e., it has no rational eigenvalues. 

\begin{lemma}\thlabel{eigenvals}
Let $T = T_\al$ be defined as in Section \ref{construction} for an irrational $\alpha$ and let $(\frac{p_k}{q_k})$ be the continued fraction convergents for $\alpha$. If $\lambda$ is an eigenvalue of $T$, then
$$\lim_{j \to \infty} \lambda^{q_j} = 1.$$
\end{lemma}

\begin{proof}
The proof we give works in a more general setting, i.e., suppose there exists $N_0$ so that for all $n>N_0,$ for every level $I$ of $C_n$ at least $1/2$ (say) comes back under $h_n$: $\mu(T^{h_n}I\cap I)>1/2\mu(I)$. The argument we give shows that $\lim \la^{h_n}=1$. Let now  $f$ be an eigenfunction of $T$. Then $f \circ T = \la f$ for some $\la\in\mathbb C$ with $|\la|=1$. For each $\eps > 0$, there is some $c$ such that the set 
$$A = \left\{ x : |f(x) - c| < \frac{\eps}{2} \right\}$$
has positive measure. Then there is a level, that we may denote by  $I_{i}$, on the $i$th column for which $\mu (A\cap I_i)> \frac{3}{4} \mu(I_i)$. Note that by the construction of $T$, for all sufficiently large $i$, at least $1/2$  of $I_i$ returns to itself under $T^{q_{i}}$: $\mu(T^{q_i}I_i\cap I_i)>1/2\mu(I_i)$. 
 So there exists $x\in A$ such that  $T^{q_i}(x) \in A$. Thus   $\frac{\eps}{2}> |f \circ T^{q_i}(x) - c| = |\lambda^{q_i}f(x) - c| $. By the triangle inequality, we have 
$$|\lambda^{q_i} - 1| < \eps.$$
Note that for that same $\eps$, since $I_i$ is partitioned into several subinterval of equal measure, there is at least one  subinterval that is more than $3/4$ full of $A$; this will be a level in $C_{i+1}$. Thus, we may conclude that for each $j \ge i$, we have
$$|\lambda^{q_j} - 1| < \eps$$
This shows that 
$$\lim_{j \to \infty} \lambda^{q_j} = 1.$$
\end{proof}

The proof of the following lemma follows ideas in Turek \cite{Tu78}. Other conditions for total ergodicity for rank-one transformations are in \cite{CrSi04, GaHi13}.

\begin{lemma}\thlabel{relprimecols}
Let $T$ be a rank-one transformation defined on a finite or infinite measure space. Suppose for every eigenvalue $\lambda$,  
$\lim_{j \to \infty} \lambda^{h_j} = 1$. If there is a subsequence $(h_{n_j})$ of the column heights $(h_n)$ that are relatively prime, then T is totally ergodic.
\end{lemma}

\begin{proof}[Proof (Turek)]
We prove the contrapositive. Suppose $T$ is not totally ergodic. Then $T$ has a rational eigenvalue, call it $\lambda$. Let $k$ be the smallest positive integer such that
$\lambda^k = 1$. Then $h_n = b_nk + r_n$ for some $b_n$  and $0 \leq r_n < k$.
Since $$|1-\lambda^{h_n}| = |1-\lambda^{b_nk+r_n}| = |1-\lambda^{r_n}|$$
and $\lim_{j \to \infty} \lambda^{h_j} = 1$, it follows that $\lim_{j \to \infty} \lambda^{r_j} = 1$.
Suppose that there existed a subsequence $(r_{n_j})$ such that $ r_{n_j}  \neq 0$. Then $|1 - \lambda^{r_{n_j}}| \geq |e^{2\pi i/k} - 1| > 0$, which would mean that $\lim_{j \to \infty} \lambda^{r_j} \neq 1$. Therefore no such subsequence can exist, so for all $n \geq N$ for some $N$, $r_n = 0$. Then $h_n = b_nk$ for all $n \geq N$ and $(h_N, h_{N+1}, ...) \geq k > 0$.
\end{proof}

\begin{thm}\thlabel{toterg}
The transformation $T_\al$ is totally ergodic.
\end{thm}

\begin{proof}
Since $q_{k+1} = a_kq_k+ q_{k-1}$, we have $(q_k, q_{k+1}) = 1$ for all $k$. Then, we may apply \thref{relprimecols} to complete the proof.
\end{proof}

\section{Uniqueness of Eigenvalues For Badly Approximable Numbers}\label{unique}

As in \thref{relprimecols}, we work with $[0, 1) \pmod{1}$ instead of $\bs^1$. Let $\al$ satisfy the conditions of \thref{mainthm}, and $T_{\alpha}$ be the transformation constructed in section \ref{construction} for $\alpha$. Suppose $\beta$ is a real number. Denote $p_{k, \beta} = [q_k\beta]$, where $[\cdot]$ is the nearest integer function. Let $\eps_{k, \beta} = |p_{k, \beta} - q_k\beta|$. We say the \textbf{address} of $\beta$ is the infinite string $p_{\beta} = p_{1, \beta}p_{2, \beta} p_{3, \beta} \cdots$. We thus recast \thref{eigenvals}:

\begin{thm} \thlabel{convzero}
If $\lambda$ is an eigenvalue of $T_{\alpha}$, then $\eps_{k, \lambda}$ converges to $0$. 
\end{thm}

\begin{thm} \thlabel{uniqueEigen} 
Suppose $\alpha$ is badly approximable. Then the only eigenvalues for $T_{\alpha}$ are the integer multiples of $\alpha$. 
\end{thm}

\begin{proof}
Let $M$ be a strict upper bound on the continued fraction coefficients for $\al$. By \thref{convzero}, there exists a $K$ so that for each $k \ge K$, we have $\eps_k < \frac{1}{4M}$. Choose $\eps = \frac{1}{4M}$. There exists a $K$ such that for each $k \ge K$, we have $\eps_k < \eps < \frac{1}{4M}$. We see that $[\frac{p_{k, \beta}}{q_k} - \frac{\eps}{q_k}, \frac{p_{k, \beta}}{q_k} + \frac{\eps}{q_k}]$ contains $\beta$, and since $\eps_{k + 1} < \eps$, by the triangle inequality, we have 
$$\left|\frac{p_{k, \beta}}{q_k} - \frac{p_{k + 1}, \beta}{q_{k + 1}}\right| \le \left|\frac{p_{k, \beta}}{q_k} - \beta\right| + \left|\beta - \frac{p_{k + 1, \beta}}{q_{k + 1}}\right| \le \frac{\eps_k}{q_k} + \frac{\eps_{k + 1}}{q_{k + 1}} < \frac{2\eps}{q_k}.$$
Thus, $(\frac{p_{k, \beta}}{q_k} - 2\frac{\eps}{q_k}, \frac{p_{k, \beta}}{q_k} + 2\frac{\eps}{q_k})$ contains exactly one rational point whose denominator divides $q_{k + 1}$ (in reduced form). By properties of continued fractions, we have
$$\left|\frac{np_k}{q_k} - \frac{np_{k + 1}}{q_{k + 1}}\right| = \frac{|n|}{q_kq_{k + 1}}$$
Now choose $n$ such that $|n| \le \frac{q_k}{2}$ and $np_k \equiv p_{k, \beta} \pmod{q_k}$. Since each ball of radius $\frac{1}{2q_{k + 1}}$ contains at most one rational point of denominator dividing $q_{k + 1}$, and $\frac{n}{q_kq_{k + 1}} \le \frac{1}{2q_{k + 1}}$, the ball centered on $\frac{p_{k, \beta}}{q_k}$ with radius $\frac{1}{2q_{k + 1}}$ has precisely one rational point of denominator dividing $q_{k + 1}$. If $|n| = \frac{q_k}{2}$, the point $\frac{np_{k + 1}}{q_{k + 1}}$ is on the boundary of the ball $B_{(2q_{k + 1})^{-1}}(\frac{np_k}{q_k})$, and
$$\left|\frac{np_{k + 1}}{q_{k + 1}} \pmod{1} - \frac{p_{k + 1, \beta}}{q_{k + 1}}\right| < \frac{1}{q_{k + 1}}.$$
But this is a contradiction. Thus, $|n| < \frac{q_k}{2}$, so $\frac{np_{k + 1}}{q_{k + 1}} \pmod{1} = \frac{p_{k + 1, \beta}}{q_{k + 1}}$. \\\\
Suppose $r > k$ and for each $\ell$ such that $r > \ell \geq k$, $\frac{p_{\ell, \beta}}{q_{\ell}} = \frac{np_{\ell}}{q_\ell} \pmod{1}$. Then since $\eps_r < \eps$, $(\frac{p_{r - 1, \beta}}{q_{r - 1}} - 2\frac{\eps}{q_{r - 1}}, \frac{p_{r - 1, \beta}}{q_{r - 1}} + 2\frac{\eps}{q_{r - 1}})$ contains exactly one rational point with denominator dividing $q_{r}$. Since $p_{r - 1, \beta} = np_{r - 1} \pmod{q_{r - 1}}$, $\frac{1}{2q_{r}} > \frac{1}{2Mq_{r - 1}}$, and $B_{(2q_r)^{-1}}(\frac{p_{r - 1, \beta}}{q_{r - 1}})$ contains at most one rational point with denominator dividing $q_r$, that ball must contain one rational point. That rational point must be $\frac{np_r}{q_r} = \frac{p_{r, \beta}}{q_r} \pmod{1}$. By induction, for each $s > k$, we have $\frac{np_{s}}{q_s} = \frac{p_{s, \beta}}{q_s} \pmod{1}$. Since $\frac{np_{s}}{q_s} \pmod{1}$ converges to $n\alpha \pmod{1}$ and $\frac{p_{s, \beta}}{q_s}$ converges to $\beta$, we have $n\alpha = \beta \pmod{1}$.  
\end{proof}

\section{Injectivity of Eigenfunction}\label{one-one}
In this section, we will establish the following result:

\begin{thm}\thlabel{injective}
For each $\alpha$, let $f_\alpha$ correspond to the eigenfunction with eigenvalue $\alpha$. Then $f_\alpha$ is injective almost everywhere.
\end{thm}

Let $\zeta_k = |q_k\alpha - p_k|$, where $(\frac{p_k}{q_k})$ are the continued fraction convergents of $\alpha$. The following standard properties of continued fractions will be used to show that the eigenfunction $f$ as constructed in section \ref{construction} is injective.

\defin{
We say that a rational number $\frac{p}{q}$ is a \textbf{best approximation} of an irrational number $\alpha$ if for $0 < q' \le q$, and $p' \in \mathbb{Z}$, then 
$$|q\alpha - p| \le |q'\alpha - p'|.$$
}

\begin{prop}
For a given irrational number $\alpha$, the best approximations of $\alpha$ are the continued fraction convergents $(\frac{p_k}{q_k})$. 
\end{prop}

\begin{cor}\thlabel{crux}
If $p, q$ are integers so that $0 < q < q_k$, then $|q\alpha - p| \ge \zeta_{k - 1}$
\end{cor}

\begin{proof}
Suppose there is some $q$ and $p$ such that $|q\alpha - p| < \zeta_{k - 1}$ and $q_{k - 1} < q < q_k$. Let $p$ and $q$ be chosen so that $q$ and $p$ are minimal. Then $\frac{p}{q}$ is a best approximation of $\alpha$, but since $q_{k - 1} < q < q_k$, $\frac{p}{q}$ is not a continued fraction convergent. 
\end{proof}

\begin{lemma}
For all $k \ge 1$, $\zeta_{k - 1} > a_k\zeta_k$.
\end{lemma}

\begin{proof}
First, observe that $\zeta_{k + 1} < \zeta_k$. This is true by the theory of continued fractions. In addition, if $\frac{p_k}{q_k} > \alpha$ then $\frac{p_{k + 1}}{q_{k + 1}} < \alpha$. Since $q_{k + 1} = a_kq_k + q_{k - 1}$, we have 

\begin{align*}
-\zeta_{k + 1}& = -\alpha q_{k + 1} + p_{k + 1} \\
&= -a_kq_k\alpha + a_kp_{k} - q_{k - 1}\alpha + p_{k - 1} \\
&= a_k\zeta_{k} - \zeta_{k - 1}\\
&< 0.
\end{align*} 
This implies that $a_k\zeta_k < \zeta_{k - 1}$ as desired. If $\frac{p_k}{q_k} < \alpha$, then 
\begin{align*}
\zeta_{k + 1} &= -\alpha q_{k + 1} + p_{k + 1} \\
&= -a_kq_k\alpha + a_kp_{k} - q_{k - 1}\alpha + p_{k - 1}\\
& = -a_k\zeta_{k} + \zeta_{k - 1}\\
& > 0.
\end{align*}

Hence $a_k\zeta_k < \zeta_{k - 1}$ as well. 
\end{proof}
\begin{rem}
The argument above shows that $a_k\zeta_k + \zeta_{k + 1} = \zeta_{k - 1}$
\end{rem}`

\begin{lemma}\thlabel{Telescope}
For each positive integer $\ell$, 
$$\sum_{i = \ell}^\infty (a_i - 1)\zeta_i = a_\ell\zeta_\ell - \sum_{i = \ell + 2}^\infty \zeta_i.$$
\end{lemma}
\begin{proof}
Rearranging the left hand side of the sum, we have
\begin{align*}
\sum_{i = \ell}^\infty (a_i - 1)\zeta_i &= a_\ell\zeta_\ell - \sum_{i = \ell}^\infty (\zeta_\ell - a_2\zeta_2) \\
&= a_\ell\zeta_\ell - \sum_{i = \ell + 2}^\infty \zeta_i
\end{align*}
by the previous remark.
\end{proof}

\begin{prop}\label{}
If $(z_k)$ is an integer sequence such that $0 \le |z_k| < a_k$, $\forall k \in \mathbb{N}$ and 
$$\sum_{k = 1}^\infty z_k\zeta_k = 0 \pmod{1},$$
then $z_k = 0$ for each $k$.

\end{prop}

\begin{proof}
Suppose not all $z_k$ are $0$. Let $\ell$ be the smallest nonnegative integer for which $z_\ell \neq 0$. Then 
\begin{align*}
\left|\sum_{k = \ell}^\infty z_k\zeta_k\right| &\ge \zeta_\ell - \sum_{k = \ell + 1}^\infty (a_k - 1)\zeta_k \\ &= \sum_{k = \ell}^\infty (\zeta_k - a_{k + 1}\zeta_{k + 1}) \\&> 0,
\end{align*}
by \thref{Telescope}. Furthermore, if $a_k = 1$, then $z_k = 0$ so we can assume that each $a_k > 1$ for $k > 1$, and we have 
\begin{align*}
    \left|\sum_{i = \ell}^\infty z_k\zeta_k\right| &\le \sum_{k = 1}^\infty (a_k - 1)\zeta_k \\&= a_1\zeta_1 - \sum_{k = 3}^\infty \zeta_k \\&< a_1\zeta_1 \\ &< 1
\end{align*}
since $\zeta_1 = \alpha - a_0 < \frac{1}{a_1}$ where $[\cdot]$ is the greatest integer function. Hence $z_k = 0$ for each $k$. 
\end{proof}

We are now ready to prove the main theorem of the section.
\begin{proof}[Proof of \thref{injective}]
If $x$ is on column $C_k$, let $\ell_{k, x}$ be the level on which $x$ lies. Let column $k_x$ be the first column that contains $x$ and $m_{i, x}$ represent the $m_{i, x}$th portion of the $\ell_{i, x}$th level that contains $x$. From careful inspection of the proof of theorem \ref{mainthm}, we see that the quantity
$$\exp\left(2\pi i \alpha\left(\ell_{k_x, x} + \sum_{i = k_x}^N m_{i,x}q_i\right)\right)$$
converges to $f_\alpha(x)$ as $N$ gets large. Suppose $f_\alpha(x) = f_\alpha(y)$. Without a loss of generality, let $k_x \ge k_y$. As above, we assume that $a_k \neq 1$ for each $k$, because otherwise, $m_{i, x} = 0$. Then 
$$\lim_{N \to \infty} \exp\left(2\pi i \alpha\left(\ell_{k_x, x} + \sum_{i = k_x}^N m_{i, x}q_i\right)\right) = \lim_{N \to \infty} \exp\left(2\pi i \alpha\left(\ell_{k_y, y} + \sum_{i = k_y}^N m_{i, y}q_i\right)\right)$$
which gives
$$\lim_{N \to \infty} \exp\left(2\pi i \alpha\left(\ell_{k_x, x} - \ell_{k_y, y} - \sum_{i = k_y}^{k_x - 1} m_{i, y}q_i + \sum_{i = k_x}^N (m_{i, x} - m_{i, y})q_i\right)\right) = 1$$
and hence
$$(\ell_{k_x, x} - \ell_{k_y, y})\alpha - \sum_{i = k_y}^{k_x - 1} (-1)^{i + 1}m_{i, y}\zeta_i + \sum_{i = k_x}^\infty (-1)^{i + 1}(m_{i, x} - m_{i, y})\zeta_i \pmod{1} = 0.$$
Note that the $(-1)^i$ are present because of  \thref{a}. Let 
$$C = \ell_{k_x, x} - \ell_{k_y, y} - \sum_{i = k_y}^{k_x - 1} m_{i, y}q_i.$$
First, we claim that if $C = 0$, then $x = y$. To prove this, if $C = 0$, then $\ell_{k_x, x} = \ell_{k_x, y}$, because if $k_x > k_y$, then $y \not\in C_{k_x} \setminus C_{k_y}$ and $x \in C_{k_x} \setminus C_{k_y}$. But that means that $\ell_{k_x, x} \neq \ell_{k_x, y}$ since level $\ell_{k_x, x}$ in column $C_{k_x}$ are all spacers but $\ell_{k_x, y}$ is not a spacer. Hence $k_x = k_y$. Since $m_{e, x}$ and $m_{e, y}$ are less than  $a_e$, we must have $|m_{e, x} - m_{e, y}| < a_e$ for all $e \ge k_x$. By the previous proposition, this must mean that $m_{e, x} = m_{e, y}$ for each nonnegative integer $e \ge k_x$. This must imply that $x = y$ since if $x \neq y$, let $z$ be an integer such that $\frac{1}{q_z} < |x - y|$. Then on column $C_z$, $x$ and $y$ must be on different levels. But $\ell_{z, x} = \ell_{z, y}$, a contradiction. \\\\
Next, we claim that if $C \neq 0$, then $|C| \le q_{k_x} - 1$ and therefore $C\alpha \pmod{1} \ge \zeta_{k_x - 1}$. To prove this, first, note that 
\begin{align*}
a_{k_x - 1}q_{k_x - 1} \le |\ell_{k_x, x}| < q_{k_x}, \\
a_{k_y - 1}q_{k_y - 1} \le |\ell_{k_y, y}| < q_{k_y}.
\end{align*}
In addition, by Formula \ref{b}, 
\begin{align*}
    0 & \le \ell_{k_y, y} + \sum_{i = k_y}^{k_x - 1} m_{i, y} q_i \\
    &\le q_{k_y} - 1 + \sum_{i = k_y}^{k_x - 1} (a_i - 1)q_i \\
    & = q_{k_y} - 1 + a_{k_x - 1}q_{k_x - 1} - q_{k_y} - \left(\sum_{i = k_y - 1}^{k_x - 2} (a_iq_i - q_{i + 1})\right) \\
    & = q_{k_y} - 1 + a_{k_x - 1}q_{k_x - 1} - q_{k_y} - \left(\sum_{i = k_y - 1}^{k_x - 3} q_i\right)\\
    & < q_{k_y} + a_{k_x - 1}q_{k_x - 1} - 1.
\end{align*}
Hence, $|C| \le q_{k_y} - 1$ or $|C| \le q_{k_x} - 1$. In either case, $|C| \le q_{k_x} - 1$. By \thref{crux}, $C\alpha \pmod{1} \ge \zeta_{k_x - 1}$. \\\\
Next, notice that if $C \neq 0$, then $\zeta_{k_x} < \frac{1}{2}$. Indeed, if $a_1 = 1$, then $k_x > 2$ since $C$ is nonzero. As a result, $\zeta_{k_x - 1} < \frac{1}{q_{k_x}} = \frac{1}{a_{k_x - 1}q_{k_x - 1} + q_{k_x - 2}} \le \frac{1}{2}$. Otherwise, 
$$\zeta_{k_x - 1} < \frac{1}{q_{k_x}} = \frac{1}{a_{k_x - 1}q_{k_x - 1} + q_{k_x - 2}} \le \frac{1}{2q_{k_x - 1} + q_{k_x - 2}} < \frac{1}{2}.$$ \\\\
Finally, if $C \neq 0$, then $C\alpha \pmod{1} \ge \zeta_{k_x}$ since for all $m$, $|q_m\alpha - p_m| \le |q\alpha - p|$ for $|q| \le q_m$ and $p \in \mathbb{Z}$ (these are properties of continued fractions). But by \thref{Telescope},
\begin{align*} 
\sum_{i = k_x}^\infty (-1)^{i + 1}(m_{i, x} - m_{i, y})\zeta_i  
&\le  \sum_{i = k_x}^\infty (a_{i} - 1)\zeta_i \\ 
&= a_{k_x}\zeta_{k_x} - \sum_{i = k_x + 2}^\infty \zeta_i \\ 
&< a_{k_x}\zeta_{k_x} \\ 
&< \zeta_{k_x - 1}  \\ 
&\le |C\alpha \pmod{1}|.
\end{align*} 
Thus, 
\[\left|C\alpha \pmod{1} - \sum_{i = k_x}^\infty (-1)^{i + 1}(m_{i, x} - m_{i, y})\zeta_i \right| > 0.\]
Now choose a suitable equivalence class for $C\alpha$ so that $|C\alpha \pmod{1}| < \frac{1}{2}$. Note that regardless of the equivalence class chosen, $|C\alpha \pmod{1}| \ge \zeta_{k_x - 1}$, so the above argument  still holds for this equivalence class. We get
$$\left|C\alpha \pmod{1} - \sum_{i = k_x}^\infty (-1)^{i + 1}(m_{i, x} - m_{i, y})\zeta_i \right| < \frac{1}{2} + \zeta_{k_x - 1} < \frac{1}{2} + \frac{1}{2} = 1$$
by our work above. Hence, 
$$C\alpha \pmod{1} - \sum_{i = k_x}^\infty (-1)^{i + 1}(m_{i, x} - m_{i, y})\zeta_i \neq 0 \pmod{1}.$$
This is a contradiction, and so  $x = y$. This completes the proof that $f_\alpha$ is injective.
\end{proof}

\begin{rem}
The best approximation property of continued fractions is crucial to our argument that the eigenfunction is injective almost everywhere.
\end{rem}

\coro{
We have the following: $f_\alpha^{-1}(\mathfrak{B}(\bs^1)) = \mathfrak{B}(X)$. 
}
\begin{proof}
The sets $\mathcal{G} = \{f_\alpha^{-1}(I): I \text{ is an interval with rational endpoints}\}$ generates $f_\alpha^{-1}(\mathfrak{B}(\bs^1))$. Since $f_\alpha$ is injective, $\mathcal{G}$ separates points so by Blackwell's theorem, it generates $\mathfrak{B}(X)$.  
\end{proof}

\coro{If $(X_\alpha, \mu)$ has infinite measure, then the pushforward measure ${f_\alpha}_*\mu$ is an infinite invariant measure for rotation by $\alpha$ and is mutually singular with respect to Lebesgue measure $m$ on $S^1$}
\begin{proof}
As $f_\alpha$ is an isomorphism, ${f_\alpha}_*\mu$ is a $\sigma$-finite infinite measure that is ergodic with respect to the irrational rotation by $\al$. It is well known (see e.g. \cite{Aa97}) that any $\sigma$-finite infinite $S$-ergodic measure is mutually singular with respect to any finite $S$-ergodic measure for an ergodic transformation $S$. Hence, ${f_\alpha}_*\mu$ is mutually singulary with respect to to Lebesgue measure.
\end{proof}

\begin{rem}
 It is not the case that if $(X, T, \mu, \mathcal{S})$ is any infinite ergodic measure preserving system with an irrational eigenvalue $\alpha$ and eigenfunction $g$, that $g_*\mu$ is mutualy singular with respect to Lebesgue measure $m$. An example of such a system is the product of any infinite measure preserving weakly mixing system and the irrational rotation by $\alpha$. This system is ergodic and has a factor to the irrational rotation by $\alpha$. The pushforward measure of that factor is a measure equivalent to the Lebesgue measure that takes the value of infinity for any positive Lebesgue measure sets and $0$ for all Lebesgue measure zero sets.
\end{rem}

\begin{rem}
There exist infinite $\sigma$-finite invariant measures for irrational rotations that are not ergodic. Let $R_\alpha$ be an irrational rotation. Then there exist infinite $\sigma$-finite nonatomic ergodic invariant measures $\mu_1$ and $\mu_2$ that are mutually singular \cite{Sc77}. Let $\mu_3=\mu_1+\mu_2$. Then $\mu_3$ is an infinite $\sigma$-finite nonatomic ergodic measure for $R_\alpha$. There is a measurable set $A$ such that $\mu_1(A)=0$ and $\mu_2(A)>0$. Let $A^*=\cup_{n=-\infty}^\infty R^n_\alpha(A)$. Then $A^*$ is $R_\alpha$ invariant but $\mu_3(A^*)>0$ and $\mu_3(S^1\setminus A^*)>0$, so $(S^1,R_\alpha,\mu_3)$ is not ergodic, though it is conservative.  This also shows that invariant infinite measures on $(S^1,R_\alpha)$ need not be rank-one.
\end{rem}



\section{Probability Preserving Factors of $T_\al$}\label{discretespectrum}

We start with a discussion on the notion of factors. Let  $(X, T, \mu, \mathcal{B})$ be a nonsingular dynamical system with $\mu$ a $\sigma$-finite, possibly infinite, measure. A \textbf{factor algebra} is a (countably generated) sub-$\sigma$-algebra $\mathcal F$ of $\mathcal{B}$ such that $T\mathcal F=\mathcal F$. A \textbf{$\sigma$-finite factor algebra} is a factor algebra that is $\sigma$-finite. For example, if $T$ is infinite measure-preserving and $\pi:X\times X\to X$ is the usual projection to the first coordinate, then $\pi^{-1}(\mathcal B)$ is a factor algebra of $T\times T$ that is not $\sigma$-finite. Also, if $T$ is infinite measure-preserving and has rational eigenvalues it will have nontrivial non-$\sigma$-finite factor algebras; for example, if $T$ is an infinite measure-preserving transformation and has an set $A$ such that $TA = A^c$  (this is the case for the Hajian--Kakutani transformation, for example), then if $\mathcal F$ is the sigma-algebra generated by $\{A, TA\}$, then $\mathcal F$ is factor algebra  that is not $\sigma$-finite.

 A transformation $T$ is said to be \textbf{prime} if  every factor algebra $\mathcal F$ satisfies $\mathcal F=\{\emptyset, X\}$ or $\mathcal F=\mathcal B$ (of course, all equalities are mod $\mu$),  and it is \textbf{prime in the measure-preserving sense} \cite{AaNa87} if  every $\sigma$-finite factor algebra  $\mathcal F$ satisfies $\mathcal F=\{\emptyset, X\}$ or $\mathcal F=\mathcal B$.

In \cite{AaNa87}, Aaronson and Nadkarni constructed a compact group rotation with an infinite invariant measure that is prime in the measure-preserving sense, but it has non-sigma-finite factors. Aaronson notes in \cite[3.1]{Aa97} that the Hajian--Kakutani transformation also satisfies these properties (and we know it has non-sigma finite factors). An infinite measure-preserving transformation that is prime was constructed in Rudolph--Silva \cite{RuSi89}; this transformation is not a group rotation as it is weakly mixing, in fact its Cartesian square is ergodic.

To work with possibly non-sigma-finite factors of an infinite measure-preserving transformation we will consider a probability measure equivalent to the infinite invariant measure; with respect to this new equivalent probability measure the transformation will be nonsingular.  Let $(X, T, \mu, \mathcal{B})$ be a nonsingular dynamical system with $\mu$ a probability  measure. We say that $(Y, S, \nu, \mathcal{C})$ is a \textbf{factor} of  $(X, T, \mu, \mathcal{B})$ (sometimes for emphasis we may call it a nonsingular factor) if there is a Borel map $\phi: X' \to Y'$ such that $\phi \circ T = S \circ \phi$ and $\nu$ is equivalent to $\phi_*\mu$, where $X' = X \setminus X_0$, $Y' = Y \setminus Y_0$ with $\mu(X_0) = \nu(Y_0) = 0$.  It is clear that then $\phi^{-1}(\mathcal C)$ is a factor algebra of $T$, and similarly one can show that factor algebras give rise to factors, see e.g.,  \cite{Ru90,Aa97}.  When $(X, T, \mu, \mathcal{B})$ is infinite measure-preserving, its factors will be with respect to the nonsingular system obtained when we take a probability measure equivalent to $\mu$, and one can see that this is independent of what equivalent probability measure we choose. 
We will  say $T$ has a probability preserving factor if there exists a probability preserving dynamical system $(Y, S, \nu, \mathcal{C})$  that is a (nonsingular) factor of $T$. 

Aaronson and Nadkarni \cite{AaNa87} write \lq\lq At this time, we know very little about the factors of non-singular group rotations. For example, we do not know if a non-trivial ergodic non-singular group rotation can have weakly mixing factors, or alternatively, have no non-trivial factors.\rq\rq   In this section we construct a nonsingular measure on an irrational rotation that has no weakly mixing factors. In \cite[Problem 5.5]{GlWe16}, Glasner and Weiss  ask  if a nonsingular system that does not admit nontrivial probability preserving factors with discrete spectrum is weakly mixing. Our example answers this in the negative. Finally, we recall a question from \cite{AaNa87}: they note that if a finite measure-preserving transformation is not weakly mixing, that transformation is prime if and only it is a rotation on a prime number of points, and ask if this true in the nonsingular case. If we knew our example was prime it would answer this in the negative. 

In this section, we will denote $\mu_\alpha$ as the measure for the transformation $T_\al$ we constructed. We now state the main theorem of this section.
\begin{thm}\thlabel{DiscreteSpectrumFactor}
If $X_\alpha$ has infinite measure and $\alpha$ is badly approximable, then every probability preserving factor of $(X_\alpha, T_\alpha,\mu_\alpha, \mathcal{B})$ is weakly mixing. 
\end{thm}


We choose a probability measure $\mu$ equivalent to $\mu_\alpha$, and  will consider the nonsingular system  $(X, T_\alpha, \mu, \mathcal{B}_\alpha)$.  By \thref{uniqueEigen}, eigenfunctions of $T_\alpha$ are of the form $f_\alpha^n$. Since $(f_\alpha)_*\mu_\alpha$ is mutually singular to Lebesgue measure $m$, $(f_\alpha)_*\mu$ is mutually singular to $m$. 
\begin{lemma} \thlabel{equivalentLebesgue}
Suppose $\phi: (X, T_\alpha, \mu, \mathcal{B}) \to (Y, S, \nu, \mathcal{C})$ is a factor map with $\nu$ equivalent to $\phi_*\mu$. If $g$ is an eigenfunction of $S$ with irrational eigenvalue $\beta$, then $g \circ \phi$ is an eigenfunction of $(X, T_\alpha, \mu, \mathcal{B})$ and $\beta = \alpha^n$, $g \circ \phi = f^n$ for some $n$. Furthermore, if $(Y, S, \nu, \mathcal{C})$ is a probability preserving system, then $g_*\nu$ is equal to the Lebesgue measure $m$ and thus $g_*\phi_*\mu$ is equivalent to $m$
\end{lemma}
\begin{proof}
We will first show that $g_*\nu$ is equal to the Lebesgue measure. Let $\lambda_1 = e^{2\pi i \beta}$. As $g(S) = \lambda_1 g$, it follows that for all Borel subsets $A$ of $S^1$, we have
$$g_*\nu(A) = \nu(g^{-1}(A)) = \nu(S^{-1}g^{-1}(A)) = \nu(g^{-1}(\lambda_1^{-1} A)) = g_*\nu(\lambda_1^{-1} A).$$
Here, we used the fact that
$$S^{-1}g^{-1}(A) = g^{-1}(\lambda_1^{-1}(A))$$
which comes from the fact that
$$S^{-1}(g^{-1}(A)) = \{x: g(S(x)) \in A\} = \{x: \lambda_1 g(x) \in A \} = \{x: g(x) \in \lambda_1^{-1}(A)\} = g^{-1}(\lambda_1^{-1}(A)).$$
As $\lambda_1$ is irrational, the Lebesgue measure is the only probability measure on $S^1$ which is $\lambda_1$-invariant so $g_*\nu = m$. As $\nu$ is equivalent to $\phi_*\mu$, it follows that $g_*\nu$ is equivalent to $g_*\phi_*\nu$ so $g_*\phi_*\nu = (g \circ \phi)^*\nu$ is equivalent to $m$. Lastly, we see that
$$g \circ \phi \circ T = g \circ S \circ \phi = \lambda_1 g \circ \phi$$
and since $g$ is the composition of Borel-measurable maps, it follows that $g$ is Borel-measurable, so $g$ is an eigenfunction of $T_\alpha$ so $\beta = \alpha^n$ and $g \circ \phi = f^n$ for some integer $n$.
\end{proof}
Let $n: S^1 \to S^1$ denote the map $x \mapsto nx \pmod{1}$.
\begin{lemma} \thlabel{measurable}
If $A$ is a Borel set, then $nA$ is also a Borel set.
\end{lemma}
\begin{proof}
Write $\mathbb{R} / \mathbb{Z} = [0, 1/n) \sqcup [1/n, 2/n) \sqcup \cdots \sqcup [(n - 1)/n, 1)$. Let $A_i = A \cap [i/n, (i + 1)/n)$. Note that
$$nA = \bigcup nA_i.$$
Next, note that $n: [i/n, (i + 1)/n) \to [0, 1)$ is a homeomorphism. In particular, $n$ is the inverse of a Borel map from $[0, 1)$ to $[i/n, (i + 1)/n)$. Hence $nA_i$ is Borel since Borel sets are preserved under preimages of Borel maps. Hence $nA$ is Borel.
\end{proof}
\begin{lemma} \thlabel{zero}
Let $n$ be a nonzero integer. Then for $A$ a $m$-measurable set in $\mathbb{S}^1$, $m(A) = 0 \iff m(nA) = 0$
\end{lemma}
\begin{proof}
Suppose $n > 0$. A similar argument works if $n < 0$. Write $\mathbb{R} / \mathbb{Z} = [0, 1/n) \sqcup [1/n, 2/n) \sqcup \cdots \sqcup [(n - 1)/n, 1)$. Let $A_i = A \cap [i/n, (i + 1)/n)$. Note that
$$nA = \bigcup nA_i.$$
If $m(A) = 0$, then $m(A_i) = 0$ and since for any subset $B$ of $[i/n, (i + 1)/n)$, $m(nB) = nm(B)$, it follows that $m(nA_i) = 0$ as well so $m(nA) = 0$. Conversely, if $m(nA) = 0$, then $m(nA_i) = nm(A_i) = 0$, so $m(A_i) = 0$ for all $i$.
\end{proof}
\begin{lemma} \thlabel{singularLebesgue}
If $\eta$ is a Borel measure on $\mathbb{S}^1$ that is mutually singular with respect to $m$, then $n_* \eta$ is not equivalent to $m$ for all integers $n$.
\end{lemma}
\begin{proof}
This is evidently true if $n = 0$. Since $\eta$ is mutually singular with respect to $m$, there exists a Borel set $A \subset S^1$ such that $\eta(A) > 0$ and $m(A) = 0$. As $A$ is Borel, by \thref{measurable}, $nA$ is Borel, so $n^{-1}(nA)$ is Borel. In particular, $nA$ is $n_*\eta$-measurable. Since $n^{-1}(nA) \supset A$, it follows that $n_*\eta(nA) = \eta(n^{-1}(n(A))) \ge \eta(A) > 0$. But by \thref{zero}, $m(nA) = 0$. Hence, $n_*\eta$ is not equivalent to $m$.
\end{proof}
\begin{proof}[Proof of \thref{DiscreteSpectrumFactor}]
By \thref{equivalentLebesgue}, to show that $(X, T_\alpha, \mu, \mathcal{B})$ has no non-weakly mixing factors, it suffices to show that $f^n_*\mu = n_*f_*\mu$ is not equivalent to $m$ for all integers $n$. Since $f_*\mu$ is mutually singular with respect to $m$, by \thref{singularLebesgue}, $n_*f_*\mu$ is not equivalent to $m$ for all integers $n$. Hence, $(X, T_\alpha, \mu, \mathcal{B})$ has no probability preserving factors with nontrivial eigenfunctions, and hence all probability preserving factors of $T_\al$ are weakly mixing.
\end{proof}

\coro {The transformation $(X_\al, T_\al, \mu_\al, \mathcal{B})$ has no nontrivial probability preserving factors with discrete spectrum.}
\begin{proof}
Any weakly mixing probability system has no eigenfunctions, and hence cannot have discrete spectrum unless it is trivial. Since all probability preserving factors of $(X_\al, T_\al, \mu_\al, \mathcal{B})$ are weakly mixing, it follows that $T_\al$ has no nontrivial probability preserving factors with discrete spectrum.
\end{proof}

We know that $T_\al$ is not weakly mixing, but by Corollary 9.7, it does not admit probability preserving factors with discrete spectrum, thus answering Problem 5.5 in \cite{GlWe16}.

\section{Rigidity}\label{badlyrigid} 
\begin{defn}
A measure-preserving transformation $T$ is \textbf{rigid} if there is a sequence $(n_k) \to \infty$ such that for all sets $A$ of finite measure
$$\lim_{k \to \infty}\mu(T^{n_k}(A)\Delta A) =0.$$
\end{defn}
\begin{defn}
A measure-preserving transformation $T$ is \textbf{partially rigid} if there is a sequence $(n_k) \to \infty$, and a constant $c  > 0$ such that for all sets $A$ of finite measure
$$\liminf_{k \to \infty}\mu(T^{n_k}(A)\cap A) \geq c\mu(A) .$$ In this case we say that $T$ has \textbf{partial rigidity constant} of at least $c$.
\end{defn}

\begin{defn}
The \textbf{centralizer} of a transformation $(X, \mu, T)$ is the set 
$$C(T) = \{S: X \to X \  | \text{ $S$ is $\mu$ measure-preserving and  } ST = TS \text{ a.e.}\}.$$
\end{defn}

In the finite case, all irrational rotations are rigid with respect to the Lebesgue measure. It has been shown that all rigid transformations have an uncountable centralizer (\cite{Ki86} for finite measure, and a similar argument works in infinite measure \cite{GKSXZ}). In this section, we explore the rigidity of $T_\alpha$.

\begin{thm}
$T_\alpha$ is partially rigid, with a partial rigidity constant of at least $\frac{1}{2}.$
\end{thm}

\begin{proof}
For convenience of notation we denote $T_\alpha$ by just $T$.
Consider the modified sequence of heights $(q_k)$ that excludes all heights $q_n$ for $n$ such that $a_n = 1$. Since the levels of the columns generate $\frk B$ by the definition of rank-one, it is sufficient to show that for all levels $I$, $$\lim_{k \to \infty} \mu(T^{q_k}(I) \cap I) \geq \frac{1}{2}\mu(I).$$ 
We have $T^{q_k}(C^j_k(n)) = C^{j+1}_k(n)$ for $0 \leq j < a_k -1$. In other words, all sub-columns of $C_k$ except for the rightmost one end up shifting over one sub-column to the right. Then $$\mu(T^{q_k}(C_k(n)) \cap C_k(n)) \geq \frac{a_k - 1}{a_k}\mu(C_k(n)).$$

Let $I$ be a level in $C_K$. Then for all $k \geq K$, $I$ is a union of levels in $C_k$. So for all $k \geq K$,
\begin{equation} \label{rigidity_factor}
    \mu(T^{q_k}(I) \cap I) \geq \frac{a_k - 1}{a_k}\mu(I) \geq \frac{1}{2}\mu(I). 
\end{equation}
\end{proof}

\begin{thm}\thlabel{rigid}
$T_\alpha$ is rigid if and only if the number of cuts $(a_k)$ is unbounded (i.e., $\alpha$ is well approximable).
\end{thm}

\begin{proof}
By \ref{rigidity_factor}, $\mu(T^{q_k}(I) \cap I) \geq \frac{a_k - 1}{a_k}\mu(I)$. If $(a_k)$ is unbounded, there exists a subsequence $(a_{k_j})$ such that $a_{k_j} \to \infty$. Then $\lim_{j \to \infty} \frac{a_{k_j} - 1}{a_{k_j}} = 1 $. So
$$\lim_{j \to \infty} \mu(T^{q_{k_j}}(I) \cap I) =\mu(I),$$
and $T$ is rigid.

On the other hand, suppose that the $a_k$ are bounded by some integer $M$ (i.e. $\alpha$ is badly approximable). Let $I$ be the base level of column $C_1$. We show that for all integers $p > 1$, $$\mu(T^{p}(I) \cap I^c) \geq \frac{1}{M^2}\mu(I),$$
which will imply that $T$ is not rigid.
Given an integer $p > 1$, consider column $C_k$ where $k$ is the integer such that $q_{k-1} < p \leq q_k$. Note that $C_k$ is composed of  $a_{k-1}$ copies of $C_{k-1}$ and $q_{k-2}$ spacers on top. Since $I$ does not come from spacers, $I$ will be distributed equally among the $a_{k-1}$ copies as a union of levels of $C_k$. Furthermore, any consecutive stack of $q_{k-1}$ levels contained among these copies will also contain $\frac{1}{a_{k-1}}$ of $I$, by similarity of the copies.

We now cut $C_k$ into $a_k$ sub-columns. The rightmost sub-column, denoted by $C_k^{a_k - 1}$, will contain $\frac{1}{a_k}$ of $I$. By the $\mathbf{n}^{\textbf{th }}$\textbf{level} of $C_k^{a_k - 1}$ we mean $C_k^{a_k - 1}(n)$. Let $S$ denote the $q_{k-2}$ top levels of $C_k^{a_k - 1}$ along with the $q_{k-1}$ new spacers we will be adding on top of this sub-column (see Figure \ref{fig2}). Then $I \cap S = \emptyset$, since $S$ is entirely composed of spacers. 
Since $q_{k-1} < p \leq q_k$, $T^{-p}(S) \setminus S$ must include at least $q_{k-1}$ consecutive levels of $C_k^{a_k - 1}$. Denote the union of these levels by $J$ (See Figure \ref{fig2}). Then $J$ contains at least $\frac{1}{a_ka_{k-1}}$ of $I$. So under $T^p$ for all $p > 1$, at least $\frac{1}{a_ka_{k-1}} \geq \frac{1}{M^2}$ of $I$ is sent to spacers and does not return ot $I$. Therefore, $T$ is not rigid.   

\end{proof}

\begin{figure}[H] 
\begin{tikzpicture}
\draw[white, fill=red!20] (5,4.2)rectangle (6,7.5);
\draw[white, fill=red!20] (0.5,7)rectangle (1,7.5);
\draw[white, fill=SeaGreen!35] (5,1.3)rectangle (6,2.3);
\draw[white, fill=SeaGreen!35] (0.5,6)rectangle (1,6.5);
\draw [ very thick] (1,0) -- (1,5);
\draw [ very thick] (2,0) -- (2,5);
\draw [ very thick] (3,0) -- (3,5);
\draw [ very thick] (5,0) -- (5,5);

\draw [blue, very thick] (0,0) -- (6,0);
\draw [blue, very thick] (0,1) -- (6,1);
\draw [blue, very thick] (0,2) -- (6,2);
\draw [blue, very thick] (0,3) -- (6,3);
\draw [red, very thick] (0,5) -- (6,5);
\draw [red, very thick] (0,4.2) -- (6,4.2);
\draw [blue, very thick] (0,3.8) -- (6,3.8);

\draw [red, very thick] (5,5.5) -- (6,5.5);
\draw [red, very thick] (5,6) -- (6,6);
\draw [red, very thick] (5,6.5) -- (6,6.5);
\draw [red, very thick] (5,7.5) -- (6,7.5);

\draw [decorate,decoration={brace,amplitude=7pt},xshift=-4pt,yshift=0pt]
(5,1.3) -- (5,2.3) node [black,midway,xshift=-0.8cm] 
{\footnotesize $q_{k-1}$};

\draw [decorate,decoration={brace,amplitude=7pt, mirror},xshift=0pt,yshift=4pt]
(0,-0.5) -- (6,-0.5) node [black,midway,yshift=-0.6cm] 
{\footnotesize $a_{k}$};

\draw [decorate,decoration={brace,amplitude=7pt},xshift=0pt,yshift=4pt]
(-0.5,-0.2) -- (-0.5,5) node [black,midway,xshift=-0.7cm] 
{\footnotesize $q_{k}$};

\draw [decorate,decoration={brace,amplitude=7pt, mirror},xshift=0pt,yshift=4pt]
(6,5) -- (6,7.5) node [black,midway,xshift=0.8cm] 
{\footnotesize $q_{k-1}$};

\draw [decorate,decoration={brace,amplitude=7pt, mirror},xshift=5pt,yshift=0pt]
(6,0.0) -- (6,1) node [black,midway,xshift=0.8cm] 
{\footnotesize $q_{k-1}$};

\draw [decorate,decoration={brace,amplitude=7pt, mirror},xshift=5pt,yshift=0pt]
(6,2) -- (6,3.0) node [black,midway,xshift=0.8cm] 
{\footnotesize $q_{k-1}$};

\draw [decorate,decoration={brace,amplitude=7pt, mirror},xshift=5pt,yshift=0pt]
(6,1) -- (6,2.0) node [black,midway,xshift=0.8cm] 
{\footnotesize $q_{k-1}$};

\draw [decorate,decoration={brace,amplitude=7pt, mirror},xshift=.7pt,yshift=0pt]
(6,4) -- (6,5) node [black,midway,xshift=0.8cm] 
{\footnotesize $q_{k-2}$};

\node (A) at (4, 4.8) [] {$\iddots$};
\node (A) at (4, 0.5) [] {$\iddots$};
\node (A) at (4, 3.5) [] {$\iddots$};
\node (A) at (5.5, 7.1) [] {$\vdots$};
\node (A) at (4, 0.5) [] {$\iddots$};
\node (A) at (4, 1.5) [] {$\iddots$};
\node (A) at (4, 2.5) [] {$\iddots$};
\node(A) at (0, 7.3) [] {$S=$};
\node(A) at (0, 6.3) [] {$J=$};

\draw [->, very thick] (8,3) -- (9,3);
\draw[white, fill=red!20] (10,4.2)rectangle (11,7.5);
\draw[white, fill=SeaGreen!35] (10,4.5)rectangle (11,5.5);

\draw [blue, very thick] (10,0) -- (11,0);
\draw [blue, very thick] (10,1) -- (11,1);
\draw [blue, very thick] (10,2) -- (11,2);
\draw [blue, very thick] (10,3) -- (11,3);
\draw [red, very thick] (10,5) -- (11,5);
\draw [red, very thick] (10,4.2) -- (11,4.2);
\draw [blue, very thick] (10,3.8) -- (11,3.8);
\draw [red, very thick] (10,5.5) -- (11,5.5);
\draw [red, very thick] (10,6) -- (11,6);
\draw [red, very thick] (10,6.5) -- (11,6.5);
\draw [red, very thick] (10,7.5) -- (11,7.5);
\end{tikzpicture}
\caption{}\label{fig2}
\end{figure}

This means that we have constructed an invariant measure for an irrational rotation that is not rigid. 
On the other hand, if it is well approximable, it is rigid and it has an uncountable centralizer.


\subsection{Observations and Questions:} 
\begin{itemize}

\item Does the centralizer of $T_\alpha$ contain $T_{\alpha'}$ for any other irrational $\alpha' \neq \alpha$? This would mean that the measure for $T_\alpha$ is equivalent to the measure for $T_{\alpha'}$.

\item Let $\nu_\alpha = \mu \circ f_\alpha^{-1}$. Consider the system $(X, \nu_{\alpha}, R_{\alpha})$ for a badly approximable $\alpha$. If   $R_\frac{\alpha}{2}$ is 
nonsingular for $\nu_\alpha$ then one can show it is measure-preserving for $R_\alpha$, so it would be in the centralizer of $R_{\alpha}$. 
This would  make  $C(R_\alpha)$ non-trivial since  $R_\frac{\alpha}{2}$ cannot be isomorphic to a power of $R_\alpha$. By \thref{uniqueEigen}, since $\alpha$ is badly approximable, the only eigenvalues for $(R_\alpha, \nu_\alpha)$ are $e^{2 \pi i n\alpha}$ for $n \in \mathbb{Z}$, so that $e^{2 \pi i \frac{\alpha}{2} }$ cannot be an eigenvalue. Similarly, $e^{2 \pi i \frac{\alpha}{2} }$ cannot be an eigenvalue for $R_{n\alpha}$ for any $n \in \mathbb{N}$, since $R_{n\alpha}$ are factors of $R_\alpha$.
However, $e^{2 \pi i \frac{\alpha}{2}}$ is an eigenvalue for $R_{\frac{\alpha}{2}}$ for all measures. Therefore, it cannot be isomorphic to any power of $R_\alpha$. 

    \item Given $\nu_\al$, is there a $\be$ so that $R_\be$ preserves $\nu_\al$?

    
\end{itemize}

\section{Type III}\label{type3}
We modify our construction to obtain  nonsingular rank-one transformations with no equivalent $\sigma$-finite $T$-invariant measure. Here the levels in the columns are intervals but not necesarily of the same length; these transformations are conservative and ergodic \cite{RuSi89}.
Given a nonsingular transformation $T$, let $\omega_n = \frac{d\mu \circ T^n}{d\mu}$ be the Radon-Nikodym derivative $T^n$.
\defin{
Given $(X, \mu, T)$, a nonsingular, conservative ergodic, $\sigma$-finite system, the \textbf{ratio set} $r(T)$ is the set of nonnegative real numbers $t$ such that for all $\epsilon > 0$ and for all measurable sets $A$ there exists $n > 0$ such that $$\mu(A \cap T^{-n}A \cap \{x:\omega_{n}(x) \in B_{\epsilon}(t)\}) > 0,$$
where $B_{\epsilon}$ is an $\eps$ neighborhood of $t$.
}

It is known that $r(T) \setminus \{0\}$ is a multiplicative subgroup of the positive real numbers, see e.g., \cite{DaSi09}. If $r(T) \neq \{1\}$ then $T$ admits no equivalent, $\sigma$-finite $T$-invariant measure, and we say it is of type III. One can further classify all type III transformations based on their ratio set. 
\begin{enumerate}
    \item Type III$_\lambda$: $r(T) = \{\lambda^n: n \in \mathbb{Z}\} \cup \{0\}$ for some $\lambda \in (0, 1)$. 
    \item Type III$_0$: $r(T) = \{0, 1\}$
    \item Type III$_1$: $r(T) = [0, \infty)$
\end{enumerate}

Note that this $\lambda$ is not the same as the $\lambda$ we used earlier to refer to our eigenvalue.

We show how to modify our construction $T_\alpha$ in various ways to make it each of these three types. While we do not include the proofs, other properties of these transformations similar to those for the infinite measure-preserving case can also be established. 

\subsection{Type III$_\lambda$}

First we modify our construction of $T_\alpha$ to obtain a type III$_\lambda$ transformation for each $0 < \lambda < 1$. This will give us another proof of Keane's result that there are infinitely many inequivalent, nonsingular measures for the rotation on the circle.

We fix an irrational $\alpha$ not of golden type, and let $a_k, q_k$ be defined as before. Next we choose some $\lambda \in (0,1)$. Let $C_1$ be the unit interval. We will still be cutting each level of $C_k$ into $a_k$ pieces, and adding $q_{k-1}$ spacers. However, not all of the cuts will have equal width. Specifically, we want the proportions of lengths of different pieces to be $\lambda$. Consider any level $I$ of $C_k$. We cut $I$ into $\lceil\frac{a_k}{2}\rceil$ pieces of length $\frac{1}{\lceil\frac{a_k}{2}\rceil + \lambda\lfloor\frac{a_k}{2}\rfloor}\mu(I)$, and into $\lfloor\frac{a_k}{2}\rfloor$ pieces of smaller length $\frac{\lambda}{\lceil\frac{a_k}{2}\rceil + \lambda\lfloor\frac{a_k}{2}\rfloor}\mu(I)$ (See Figure \ref{III}). Note that for $a_k$ even, $\lceil\frac{a_k}{2}\rceil = \lfloor\frac{a_k}{2}\rfloor = \frac{a_k}{2}$, while for $a_k$ odd, $\lceil\frac{a_k}{2}\rceil = \frac{a_k+1}{2}$ and $\lfloor\frac{a_k}{2}\rfloor = \frac{a_k -1}{2}$.

\begin{figure}[H] 
\begin{tikzpicture}

\draw [ very thick] (0,0) -- (6,0);
\draw [red, very thick] (4,-0.3) -- (4,0.3);

\draw [blue, very thick] (0.8,-0.1) -- (0.8,0.1);
\draw [blue, very thick] (1.6,-0.1) -- (1.6,0.1);
\node (A) at (2.4, 0.2) [] {$\dots$};
\draw [blue, very thick] (3.2,-0.1) -- (3.2,0.1);

\draw [blue, very thick] (4.2,-0.1) -- (4.2,0.1);
\draw [blue, very thick] (4.4,-0.1) -- (4.4,0.1);
\draw [blue, very thick] (4.6,-0.1) -- (4.6,0.1);
\node (A) at (5.2, 0.2) [] {$\dots$};
\draw [blue, very thick] (5.8,-0.1) -- (5.8,0.1);

\draw [decorate,decoration={brace,amplitude=7pt, mirror},xshift=0pt,yshift=4pt]
(0,-0.5) -- (4,-0.5) node [black,midway,yshift=-0.6cm] 
{\footnotesize $\lceil \frac{a_{k}}{2} \rceil$};

\draw [decorate,decoration={brace,amplitude=7pt, mirror},xshift=0pt,yshift=4pt]
(4,-0.5) -- (6,-0.5) node [black,midway,yshift=-0.6cm] 
{\footnotesize $\lfloor \frac{a_{k}}{2} \rfloor$};

\end{tikzpicture}
\caption{}\label{III}
\end{figure}

Let all the spacers have the same length as the rightmost piece of the top level. Define $T_{\alpha, \lambda}$ so that when applied to a level, it acts as the unique affine map taking it to the level above. The Radon-Nikodym derivative $\frac{d\mu T}{d\mu}$ on a level is a measure of the contraction or expansion that takes place between it and the level above. It is clear that these derivatives are powers of $\lambda$. 

The above proofs that $T_\alpha$ is totally ergodic (\thref{toterg}), and that the eigenfunction for $\lambda$ is injective (\thref{injective}) still hold. 

We now show that this modified construction is actually of type III$_\lambda$. 

\begin{thm}
$T := T_{\alpha, \lambda}$ is of type III$_\lambda$.
\end{thm}

\begin{proof}
It is clear that $\omega_{n}(x) \in \{\lambda^n\}$. Therefore, the ratio set $r(T) \in \{\lambda^n\} \cup \{0\}$. So it is sufficient for us to show that $\lambda$ belongs to $r(T)$. 
Given a measurable set $A$, we want to find a level sufficiently full of $A$ so that at least one cut from the first $\lceil\frac{a_k}{2}\rceil$ pieces and at least one cut from the second $\lfloor\frac{a_k}{2}\rfloor$ pieces are both more than $1/2$-full of $A$. Denote these two cuts by $I_1$ and $I_2$ respectively. Then $I_1$ and $I_2$ will both be levels in $C_{k+1}$, so that there exists some $n$ such that $T^n(I_1) = I_2$. Then $B = I_1 \cap A \cap T^{-1}(I_2 \cap A)$ has positive measure, and $\omega_n(x) = \lambda$ for $x \in B$. Therefore, $\lambda \in r(T)$ as desired.

If we choose a level $I$ that is at least $\frac{\lceil\frac{a_k}{2}\rceil + (\frac{3}{4})\lfloor\frac{a_k}{2}\rfloor \lambda}{\lceil\frac{a_k}{2}\rceil + \lambda\lfloor\frac{a_k}{2}\rfloor}$-full of $A$ then both the first $\frac{\lceil\frac{a_k}{2}\rceil}{\lceil\frac{a_k}{2}\rceil + \lambda\lfloor\frac{a_k}{2}\rfloor}$ and the second $\frac{\lfloor\frac{a_k}{2}\rfloor\lambda}{\lceil\frac{a_k}{2}\rceil + \lambda\lfloor\frac{a_k}{2}\rfloor}$ of the level will be at least $3/4$-full of $A$. This is because the first $\frac{\lceil\frac{a_k}{2}\rceil}{\lceil\frac{a_k}{2}\rceil + \lambda\lfloor\frac{a_k}{2}\rfloor}$ of the level is always larger than the the second $\frac{\lfloor\frac{a_k}{2}\rfloor\lambda}{\lceil\frac{a_k}{2}\rceil + \lambda\lfloor\frac{a_k}{2}\rfloor}$.

In the case where $a_k$ is even, we end up with the first $\frac{a_k}{2}$ pieces taking up $\frac{1}{1+\lambda}$ of the level, and the second $\frac{a_k}{2}$ of the pieces taking up $\frac{\lambda}{1+ \lambda}$ of the level. In the case where $a_k$ is odd, we have that the first $\frac{a_k+1}{2}$ pieces take up $\frac{1}{1+(\frac{a_k-1}{a_k+1})\lambda}$ of the level, and that the second $\frac{a_k-1}{2}$ pieces take up $\frac{(\frac{a_k-1}{a_k+1})\lambda}{1+(\frac{a_k-1}{a_k+1})\lambda}$ of the level. We can see that $$\max_{n \text{ odd } > 1} \frac{1}{1+(\frac{n-1}{n+1})\lambda} = \frac{1}{1+\frac{1}{2}\lambda} > \frac{1}{1+\lambda},$$
which occurs for $n=3$. 

Therefore, the first $\lceil\frac{a_k}{2}\rceil$ cuts take up at most $\frac{1}{1+\frac{1}{2}\lambda}$ of $I$. 

As a result, if we choose a level such that $I$ is at least $\frac{1+ (\frac{3}{4})(\frac{1}{2})\lambda}{1+\frac{1}{2}\lambda}$-full of $A$ then both the first $\frac{\lceil\frac{a_k}{2}\rceil}{\lceil\frac{a_k}{2}\rceil + \lambda\lfloor\frac{a_k}{2}\rfloor}$ and the second $\frac{\lfloor\frac{a_k}{2}\rfloor\lambda}{\lceil\frac{a_k}{2}\rceil + \lambda\lfloor\frac{a_k}{2}\rfloor}$ of the level will be at least $3/4$-full of $A$, independent of what $a_k$ is. Therefore, there must be at least one cut in the first $\lceil\frac{a_k}{2}\rceil$ cuts and one cut in the second $\lfloor\frac{a_k}{2}\rfloor$ cuts that are more than $1/2$-full of $A$, as desired. We must also take care to choose $C_k$ so that $a_k \neq 1$, so that there are actually cuts to consider. This is always possible since $\alpha$ is not of golden type.

\end{proof}

\subsection{Type III$_0$}
For $\alpha$ well approximable, we can also modify our construction to make it type III$_0$. To do so, we choose a subsequence $a_{k_j}$ of the cuts such that $a_{k_j} \to \infty$. For all $C_k$ where $k$ is not in this subsequence, we simply cut the levels into pieces of equal length. However, for $k$ in our sequence, we cut as follows: Let the first cut take up half of the level, and divide the rest of the cuts equally among the second half of the level. This is as in Example 6.3 in \cite{DaSi09}.
\subsection{Type III$_1$}
In order for a transformation to be of type III$_1$, it is sufficient to ensure that $\lambda, \beta \in r(T)$ for two real numbers $\lambda$ and $ \beta $ in $(0, 1)$ that are rationally independent, see e.g., \cite{DaSi09}. So we construct our transformation as we do in the type III$_\lambda$ case with the following modification: We alternate, in even and odd times, between cutting into proportions of  $\lambda$ and $\beta$ on the columns where $a_k \neq 1$.

\bibliographystyle{plain}
\bibliography{ErgodicBibMasterAug12018}

\begin{thebibliography}{10}

\bibitem{Aa97}
J.~Aaronson.
\newblock {\em An introduction to infinite ergodic theory}, volume~50 of {\em
  Mathematical Surveys and Monographs}.
\newblock American Mathematical Society, Providence, RI, 1997.

\bibitem{AaNa87}
Jon Aaronson and Mahendra Nadkarni.
\newblock {$L_\infty$} eigenvalues and {$L_2$} spectra of nonsingular
  transformations.
\newblock {\em Proc. London Math. Soc. (3)}, 55(3):538--570, 1987.

\bibitem{AdSi18}
Terrence~M. Adams and Cesar~E. Silva.
\newblock Weak mixing for infinite measureinvertible transformations.
\newblock {\em Ergodic Theory and Dynamical Systems in theirInteractions with
  Arithmetics and Combinatorics, Lecture Notesin Mathematics 2213,}, pages
  327--349, 2018.

\bibitem{AOW85}
Pierre Arnoux, Donald~S. Ornstein, and Benjamin Weiss.
\newblock Cutting and stacking, interval exchanges and geometric models.
\newblock {\em Israel J. Math.}, 50(1-2):160--168, 1985.

\bibitem{BSSSW15}
Francisc Bozgan, Anthony Sanchez, Cesar~E. Silva, David Stevens, and Jane Wang.
\newblock Subsequence bounded rational ergodicity of rank-one transformations.
\newblock {\em Dyn. Syst.}, 30(1):70--84, 2015.

\bibitem{CrSi04}
Darren Creutz and Cesar~E. Silva.
\newblock Mixing on a class of rank-one transformations.
\newblock {\em Ergodic Theory Dynam. Systems}, 24(2):407--440, 2004.

\bibitem{DaSi09}
Alexandre~I. Danilenko and Cesar~E. Silva.
\newblock Ergodic theory: non-singular transformations.
\newblock In {\em Mathematics of complexity and dynamical systems. {V}ols.
  1--3}, pages 329--356. Springer, New York, 2012.

\bibitem{dJ76a}
Andr\'es del Junco.
\newblock Stacking transformations and {D}iophantine approximation.
\newblock {\em Illinois J. Math.}, 20(3):494--502, 1976.

\bibitem{dJ76b}
Andr\'es del Junco.
\newblock Transformations with discrete spectrum are stacking transformations.
\newblock {\em Canad. J. Math.}, 28(4):836--839, 1976.

\bibitem{dJ77}
Andr\'es del Junco.
\newblock A transformation with simple spectrum which is not rank one.
\newblock {\em Canad. J. Math.}, 29(3):655--663, 1977.

\bibitem{ElNa16}
E.~H. El~Abdalaoui and M.~G. Nadkarni.
\newblock A non-singular transformation whose spectrum has {L}ebesgue component
  of multiplicity one.
\newblock {\em Ergodic Theory Dynam. Systems}, 36(3):671--681, 2016.

\bibitem{Fe97}
S{\'e}bastien Ferenczi.
\newblock Systems of finite rank.
\newblock {\em Colloq. Math.}, 73(1):35--65, 1997.

\bibitem{FGHSW}
Matthew Foreman, Su~Gao, Aaron Hill, Cesar~E. Silva, and Benjamin Weiss.
\newblock Rank-one transformations, odometers, and finite factors.
\newblock {\em ArXiv:1910.12645v2}, 2019.

\bibitem{GKSXZ}
Johan Gaebler, Alexander Kastner, Cesar~E. Silva, Xiaoyu Xu, and Zirui Zhou.
\newblock Partially bounded transformations have trivial centralizers.
\newblock {\em Proc. Amer. Math Soc.}, to appear.

\bibitem{GaHi13}
Su~Gao and Aaron Hill.
\newblock Bounded rank-1 transformations.
\newblock {\em J. Anal. Math.}, 129:341--365, 2016.

\bibitem{GlWe16}
Eli Glasner and Benjamin Weiss.
\newblock Weak mixing properties for non-singular actions.
\newblock {\em Ergodic Theory Dynam. Systems}, 36(7):2203--2217, 2016.

\bibitem{HMP86}
B.~Host, J.-F. M\'{e}la, and F.~Parreau.
\newblock Analyse harmonique des mesures.
\newblock {\em Ast\'{e}risque}, (135-136):261, 1986.

\bibitem{Iw94}
A.~Iwanik.
\newblock Cyclic approximation of irrational rotations.
\newblock {\em Proc. Amer. Math. Soc.}, 121(3):691--695, 1994.

\bibitem{Ke71}
Michael Keane.
\newblock Sur les mesures quasi-ergodiques des translations irrationnelles.
\newblock {\em C. R. Acad. Sci. Paris S\'er. A-B}, 272:A54--A55, 1971.

\bibitem{Ki86}
Jonathan King.
\newblock The commutant is the weak closure of the powers, for rank-{$1$}
  transformations.
\newblock {\em Ergodic Theory Dynam. Systems}, 6(3):363--384, 1986.

\bibitem{MT06}
Steven~J. Miller and Ramin Takloo-Bighash.
\newblock {\em An invitation to modern number theory}.
\newblock Princeton University Press, Princeton, NJ, 2006.
\newblock With a foreword by Peter Sarnak.

\bibitem{Na98}
M.~G. Nadkarni.
\newblock {\em Spectral theory of dynamical systems}.
\newblock Birkh\"auser Advanced Texts: Basler Lehrb\"ucher. [Birkh\"auser
  Advanced Texts: Basel Textbooks]. Birkh\"auser Verlag, Basel, 1998.

\bibitem{Ru90}
Daniel~J. Rudolph.
\newblock {\em Fundamentals of measurable dynamics}.
\newblock Oxford Science Publications. The Clarendon Press, Oxford University
  Press, New York, 1990.
\newblock Ergodic theory on Lebesgue spaces.

\bibitem{RuSi89}
Daniel~J. Rudolph and Cesar~E. Silva.
\newblock Minimal self-joinings for nonsingular transformations.
\newblock {\em Ergodic Theory Dynam. Systems}, 9(4):759--800, 1989.

\bibitem{Sc77}
Klaus Schmidt.
\newblock Infinite invariant measures on the circle.
\newblock In {\em Symposia {M}athematica, {V}ol. {XXI} ({C}onvegno sulle
  {M}isure su {G}ruppi e su {S}pazi {V}ettoriali, {C}onvegno sui {G}ruppi e
  {A}nelli {O}rdinati, {INDAM}, {R}ome, 1975)}, pages 37--43. 1977.

\bibitem{Si08}
C.~E. Silva.
\newblock {\em Invitation to {E}rgodic {T}heory}, volume~42 of {\em Student
  Mathematical Library}.
\newblock American Mathematical Society, Providence, RI, 2008.

\bibitem{Tu78}
Richard~J. Turek.
\newblock An approach to completely ergodic transformations and a stacking
  method.
\newblock {\em Rocky Mountain J. Math.}, 8(4):755--758, 1978.

\end{thebibliography}

\end{document}